\documentclass[11pt,reqno,a4paper]{amsart}

\usepackage{graphicx}
\usepackage{asymptote}
\usepackage{asypictureB}
\usepackage{enumitem}
\usepackage{caption}
\usepackage{thmtools}
\usepackage{comment}
\usepackage{setspace}
\setlist[enumerate]{font={\upshape}, label=\arabic*., leftmargin=2.5em}
\setlist[itemize]{leftmargin=2.5em}
\setlist[description]{leftmargin=\parindent, 
	itemsep=3pt
}

\newlist{equivlist}{enumerate}{1}
\setlist[equivlist]{font={\upshape}, label=(\roman*)}
\usepackage{tikz}
\usetikzlibrary{matrix,intersections,calc}
\tikzset{ 
	table/.style={
		matrix of nodes,
		nodes={rectangle,text width=1.75em,align=center},
		text depth=1.25ex,
		text height=2.5ex,
		nodes in empty cells
	}
}
\usepackage{mathtools}
\usepackage{amsfonts}
\usepackage{amssymb}
\usepackage{cite} 
\usepackage[sort,compress,numbers]{natbib}

\usepackage[pdftex
,pagebackref
]{hyperref}
\hypersetup{
	colorlinks=true,
	allcolors=blue
}
\usepackage[capitalize]{cleveref}
\usepackage[margin=1in]{geometry}
\newtheorem{theorem}{Theorem}[section]
\newtheorem{lemma}[theorem]{Lemma}
\newtheorem{claim}{Claim}[theorem]
\Crefname{claim}{Claim}{Claims}
\newlist{lemlist}{enumerate}{1}
\setlist[lemlist]{font={\upshape}, label={\upshape(\alph*)},ref={\thelemma(\alph*)},leftmargin=*}

\newtheorem{conjecture}[theorem]{Conjecture}
\Crefname{conjecture}{Conjecture}{Conjectures}

\expandafter\let\expandafter\oldproof\csname\string\proof\endcsname
\let\oldendproof\endproof
\renewenvironment{proof}[1][\proofname]{%
	\oldproof[\normalfont\bfseries #1]%
}{\oldendproof}
\newenvironment{subproof}[1][\normalfont\it Subproof]{%
	\begin{proof}[#1]%
	}{%
	\end{proof}%
}

\newcommand{\dd}{\textquotedblleft}
\newcommand{\ee}{\textquotedblright}

\newcommand{\mac}{\mathcal}

\newcommand{\eps}{\varepsilon}

\newcommand{\nin}{\notin}

\renewcommand{\subset}{\subseteq}

\newcommand{\erh}{Erd\H{o}s-Hajnal}
\newcommand{\LL}{,\ldots,}

\newcommand{\poly}{\operatorname{poly}}

\DeclarePairedDelimiter\abs{\lvert}{\rvert}%
\DeclarePairedDelimiter\ceil{\lceil}{\rceil}%
\DeclarePairedDelimiter\floor{\lfloor}{\rfloor}%
\makeatletter
\newcommand{\leqnomode}{\tagsleft@true}
\newcommand{\reqnomode}{\tagsleft@false}
\makeatother

\date{October 25, 2023; revised January 27, 2026}
\begin{document}
	\title{Induced subgraph density. VII. The five-vertex path}
	\author{Tung Nguyen}
	\address{Princeton University, Princeton, NJ 08544, USA}
	\curraddr{Mathematical Institute and Christ Church, University of Oxford, UK}
	\email{\href{mailto:tunghn@math.princeton.edu}
		{tunghn@math.princeton.edu}}
	\author{Alex Scott}
	\address{Mathematical Institute,
		University of Oxford,
		Oxford OX2 6GG, UK}
	\email{\href{mailto:scott@maths.ox.ac.uk}{scott@maths.ox.ac.uk}}
	\author{Paul Seymour}
	\address{Princeton University, Princeton, NJ 08544, USA}
	\email{\href{mailto:pds@math.princeton.edu}{pds@math.princeton.edu}}
	\thanks{The first author was supported by AFOSR grant FA9550-22-1-0234, NSF grant DMS-2154169, a Porter Ogden Jacobus Fellowship, a Titchmarsh Research Fellowship, and a Christ Church Research Centre Grant.
		The second author was supported by EPSRC grant EP/X013642/1.
		The third author was supported by AFOSR grant FA9550-22-1-0234 and NSF grant DMS-2154169.}
	\subjclass{05C35, 05C55, 05C69, 05C75}
	\begin{abstract}
		We prove the \erh{} conjecture for the five-vertex path $P_5$; that is, there exists $c>0$ such that every $n$-vertex graph with no induced $P_5$ has a clique or stable set of size at least $n^c$.
		This completes the verification of the \erh{} conjecture for all five-vertex graphs.
		Our methods combine probabilistic and structural ideas with the iterative sparsification framework introduced in the third and fourth papers in the series.
	\end{abstract}
	\maketitle
	\section{Introduction}
	All graphs in this paper are finite and with no loops or parallel edges.
	For graphs $G,H$, a {\em copy} of $H$ in $G$ is an injective map $\varphi\colon V(H)\to V(G)$ satisfying $uv\in E(H)$ if and only if $\varphi(u)\varphi(v)\in E(G)$, for all $u,v\in V(H)$;
	and $G$ is {\em $H$-free} if there is no copy of $H$ in $G$.
	A celebrated conjecture of Erd\H{o}s and Hajnal from 1977~\cite{MR599767, MR1031262} says:
	\begin{conjecture}
		\label{conj:eh}
		For every graph $H$, there exists $\tau>0$ such that every $n$-vertex $H$-free graph has a clique or stable set of size at least $n^\tau$.
	\end{conjecture}
	Let us say that a graph $H$ {\em satisfies the \erh{} conjecture} if there exists $\tau>0$ such that every $n$-vertex $H$-free graph has a clique or stable set of size at least $n^\tau$. 
	Thus $H$ satisfies the \erh{} conjecture if and only if $\overline H$ does, where $\overline H$ denotes the complement of $H$.
	
	Erd\H{o}s and Hajnal~\cite{MR1031262} themselves proved that the conjecture holds
	for all graphs $H$ with at most four vertices, and Gy\'arf\'as~\cite{MR1425208} brought attention to the five-vertex case over 25 years 
	ago (and this was reiterated in~\cite{MR3150572,MR3497267,scott2022}).
	But even for five-vertex graphs the conjecture has been extremely resistant.
	By a theorem of Alon, Pach, and Solymosi~\cite{MR1832443} that the class of graphs satisfying the \erh{} conjecture is closed under vertex-substitution,
	the problem (for five-vertex graphs) reduces to showing \cref{conj:eh} for three graphs with five vertices: the bull (obtained from the four-vertex path by adding a new vertex adjacent to the two middle vertices),
	the five-cycle $C_5$, and the five-vertex path $P_5$ (or equivalently, the {\em house}~$\overline{P_5}$).
	In 2008, Chudnovsky and Safra~\cite{MR2462320} showed that the bull satisfies the conjecture (see~\cite{MR4563865,density4} for two new 
	proofs using different methods);
	and in 2023, Chudnovsky, Scott, Seymour, and Spirkl~\cite{MR4563865} showed that $C_5$ satisfies the conjecture. But until now
	the final case, $P_5$, has remained open.
	
	There have been several  successively stronger partial results for $P_5$. Let $G$ be $P_5$-free, with $n$ vertices, and let $m$
	be the size of its largest clique or stable set. There exists $c>0$ (not depending on $G$) such that: 
	\begin{itemize}
		\item $m\ge 2^{c(\log n)^{1/2}}$, by a general theorem of Erd\H{o}s and Hajnal~\cite{MR1031262}. (This bound is not special to $P_5$, and the same holds with any excluded induced subgraph $H$.) More recently the bound was improved to
		$m\ge 2^{c(\log n\log \log n)^{1/2}}$ (again, for any $H$)~\cite{density1}.
		\item $m\ge 2^{c(\log n)^{2/3}}$, by a result of Blanco and Buci\'c~\cite{bb2022}.
		\item $m\ge 2^{(\log n)^{1-o(1)}}$ (this in fact holds when $H$ is a path of any length)~\cite{density5}.
	\end{itemize}
	But finally we can prove the full conjecture for $P_5$:
	\begin{theorem}
		\label{thm:p5eh}
		$P_5$ satisfies the \erh{} conjecture.
	\end{theorem}
	As in some previous papers of this series, our main result is in a more general form and says that $P_5$ actually satisfies the polynomial form of a theorem of R\"odl.  To discuss this we need some further definitions and results.
	For a graph $G$, $\abs G$ denotes the number of vertices of $G$.
	For $\eps>0$, we say that $G$ is {\em $\eps$-sparse} if its maximum degree is at most $\eps\abs G$, and {\em $\eps$-restricted} if one of $G,\overline G$ is $\eps$-sparse. We also say $S\subseteq V(G)$ is {\em $\eps$-restricted} if $G[S]$ is {\em $\eps$-restricted}.
	
	R\"odl's theorem~\cite{MR837962} states that:
	\begin{theorem}
		\label{thm:rodl}
		For every $\eps\in(0,1/2)$ and every graph $H$, there exists $\delta>0$ such that every $H$-free graph $G$ has an $\eps$-restricted induced subgraph with at least $\delta\abs G$ vertices.
	\end{theorem}
	The original proof of R\"odl used the regularity lemma and gave tower-type dependence of $\delta$ on $\eps$.
	Fox and Sudakov~\cite{MR2455625} provided a different proof that gives the better bound $\delta=2^{-d(\log\frac1\eps)^2}$ (here $d>0$ is some constant depending on $H$ only);
	and currently the best known bound for this theorem is $\delta=2^{-d(\log\frac1\eps)^2/\log\log\frac1\eps}$, obtained in~\cite{density1}.
	Fox and Sudakov~\cite{MR2455625} also made the much stronger conjecture that $\delta$ can be taken to be a power of $\eps$.
	Accordingly, let us say that a graph $H$ has the {\em polynomial R\"odl property} if there exists $d>0$ such that for every $\eps\in(0,1/2)$, every $H$-free graph $G$ has an $\eps$-restricted induced subgraph with at least $\eps^d\abs G$ vertices.
	It is not hard to check that every graph with the polynomial R\"odl property satisfies the \erh{} conjecture.
	The Fox--Sudakov conjecture is then the following:
	\begin{conjecture}
		\label{conj:polyrodl}
		Every graph $H$ has the polynomial R\"odl property.
	\end{conjecture}
	
	As mentioned above, the main result of this paper says that \cref{conj:polyrodl} holds for $H=P_5$, which contains \cref{thm:p5eh}:
	\begin{theorem}
		\label{thm:main}
		$P_5$ has the polynomial R\"odl property.
	\end{theorem}
	
	It was recently shown by Buci\'c,
	Fox and Pham \cite{bfp} that for every $H$, $H$ satisfies the Fox--Sudakov conjecture if and only if $H$ satisfies the 
	Erd\H os-Hajnal conjecture.  Thus \cref{thm:p5eh} and \cref{thm:main} are equivalent.  However, for our proof method, it is convenient 
	to prove the result in the stronger polynomial R\"odl form, as it allows us to approach the result through a process where we iteratively decrease $\eps$.
	
	Since this paper was submitted for publication, the first author has proved a much stronger result, that there is a positive integer $k$ for which every $P_5$-free graph has chromatic number at most the $k$th power of its clique number~\cite{2025pc5}. That result implies \cref{thm:p5eh}, but its proof uses \cref{thm:p5eh} as a critical input and adapts some of the ideas introduced here, and so does not supersede this paper.

	
	\section{A few definitions}
	\label{sec:defs}
	
	It will be useful to gather together some definitions that we will use throughout the paper.
	If $k\ge 1$ is an integer, we define $[k]:=\{1,2,\ldots,k\}$.
	If $G$ is a graph, and $A,B\subset V(G)$ are disjoint, we say that $(A,B)$ is {\em anticomplete} in $G$ (or $A$ is {\em 
		anticomplete} to $B$ in $G$) if there is no edge between $A,B$;
	and we  say that $(A,B)$ is {\em complete} in $G$ (or $A$ is {\em complete} to $B$ in $G$) if $(A,B)$ is anticomplete in $\overline G$.
	A vertex $v\in V(G)\setminus A$ is {\em mixed} on $A$ if it has both a neighbour and a nonneighbour in $A$.

	A {\em blockade} in $G$ is a sequence $\mac B=(B_1,\ldots,B_k)$ of disjoint (and possibly empty) subsets of $V(G)$; its {\em length} is $k$ and its {\em width} is $\min_{i\in[k]}\abs{B_i}$.
	For $\ell,w\ge0$, $\mac B$ is an {\em $(\ell,w)$-blockade} if it has length at least $\ell$ and width at least $w$.
	We say that the blockade $\mac B$ is 
	\begin{itemize}
		\item  {\em pure} if, for all distinct $i, j$, the pair $(B_i,B_j)$ is either complete or anticomplete;
		\item  {\em complete} in $G$ if $(B_i,B_j)$ is complete in $G$ for all distinct $i,j$; and 
		\item {\em anticomplete} in $G$ if $(B_i,B_j)$ is anticomplete in $G$ for all distinct $i,j$.
	\end{itemize}
	Note that being pure is a much weaker property than being complete or anticomplete, as there might be a mixture of complete and anticomplete pairs.
	In general, complete or anticomplete blockades that are long and wide are highly desirable.  Pure blockades are also helpful, but typically require further treatment. 
	
	For $x>0$ and disjoint $A,B\subset V(G)$, we say that $B$ is {\em $x$-sparse} to $A$ in $G$ if every vertex in $B$ has at most $x\abs A$ neighbours in $A$.
	For $A,B\ne\emptyset$, the {\em edge density} between $A,B$ in $G$ is the number of edges between $A,B$ in $G$ divided by $\abs A\abs B$;
	and we say that $(A,B)$ is {\em weakly $x$-sparse} in $G$ if the edge density between $A,B$ in $G$ is at most $x$.
	A blockade $\mac B=(B_1,\ldots,B_k)$ in $G$ is {\em $x$-sparse} in $G$ if $B_j$ is $x$-sparse to $B_i$ in $G$ for all $i,j\in[k]$ with $i<j$.

	\section{Some proof ideas}
	\label{sec:sketch}

	A class of graphs is {\em hereditary} if it is closed under taking induced subgraphs.
	A hereditary class $\mathcal G$ of graphs has the {\em Erd\H os-Hajnal property} if there is some $\tau>0$ such that every 
	$G\in\mathcal G$ has a clique or stable set of size at least $|G|^\tau$.  Thus, a graph $H$ satisfies the Erd\H os-Hajnal conjecture if and only if the class of all $H$-free graphs has the Erd\H os-Hajnal property. The goal of this paper is to prove that the class of $P_5$-free graphs has the Erd\H os-Hajnal property.
	
	One approach to proving that a class has the Erd\H os-Hajnal property is to show that we can always find a large complete or anticomplete pair of sets of vertices in each graph in the class.  More precisely, we say that $\mathcal G$ has the {\em strong Erd\H os-Hajnal 
		property} if there is $c>0$ such that, for every $G\in\mathcal G$ with at least two vertices there are disjoint $A,B\subset V(G)$ such that $|A|,|B|\ge c|G|$ and $A,B$ are either complete or anticomplete (in other words, there is a pure blockade of length 2 and linear width).  
	
	It is straightforward to show that the strong Erd\H os-Hajnal property implies the Erd\H os-Hajnal property (see \cite{alonplus, foxpach}). The strong Erd\H os-Hajnal property has received significant attention.
	For example, the following was proved in~\cite{MR4170220} (see \cite{pure2} for another example):
	\begin{theorem} \label{pure1a}
		For every forest $H$, there exists $c>0$ such that if $G$ is $H$-free and $\overline H$-free and $|G|\ge 2$, 
		then there exist disjoint $A,B\subseteq V(G)$ with $|A|,|B|\ge c|G|$ such that $A,B$ are complete or anticomplete. If 
		neither of $H, \overline{H}$ is a forest,
		there is no such $c$.
	\end{theorem}
	In fact, this result characterizes when a hereditary class defined by a finite number of excluded induced subgraphs has the 
	strong Erd\H os-Hajnal property: if and only if we exclude both a forest and the complement of a forest (see~\cite{MR4170220} for further discussion).  Note that if we exclude only $P_5$ then we do not obtain the strong Erd\H os-Hajnal property.
	
	An important observation of Tomon \cite{tomon} is that longer blockades  can have smaller blocks and still be useful.  We say that a hereditary class $\mathcal G$ has the {\em quasi-Erd\H os-Hajnal property} if there is $c$ such that for every 
	$G\in \mathcal G$ there exists $k\ge 2$ such that $G$ has a complete or anticomplete blockade of length $k$ and width at least $|G|/k^c$.  
	Note that the length $k$ is allowed to depend on $G$ (indeed, if we could take $k$ to be a constant then $\mathcal G$ would have 
	the strong Erd\H os-Hajnal property); but it is important that the width of the blockade depends polynomially on the length. It is 
	straightforward to show that the quasi-Erd\H os-Hajnal property implies the Erd\H os-Hajnal property.  The reverse implication also holds, as the clique or stable set of size $|G|^\eps$ guaranteed by the Erd\H os-Hajnal property provides a complete or anticomplete blockade of length at least $|G|^\eps$ and width 1 (so we can take $c=1/\eps$).

	Proving that a class has the quasi-Erd\H os-Hajnal property has been a helpful approach (see \cite{pt,MR4563865}). It can also 
	sometimes be combined with other approaches: we can try to show that we get either a blockade with the required polynomial dependence, or 
	else some other good structure (for example, this was one part of the argument in~\cite{density6}).  However, it is not in general clear how to show that a class has the quasi-Erd\H os-Hajnal property.
	
	While long complete or anticomplete blockades are good for us, they might not exist, and instead, we must make do with blockades 
	that have weaker properties.  There are a number of different possibilities (several have been used in papers from this series), and the challenge is to prove the existence of blockades that are sufficiently long and sufficiently restricted to prove a strong result.
	
	For example, consider the effects of excluding a tree.  The main result in~\cite{MR4170220} was in fact deduced from the following stronger `one-sided' result:
	\begin{theorem} \label{pure1}
		For every forest $H$, there exists $c>0$ such that if $G$ is an $H$-free, $c$-sparse graph  with $|G|\ge 2$, then there exist disjoint $A,B\subseteq V(G)$ with $|A|,|B|\ge c|G|$ such that $A,B$ are anticomplete. If $H$ is not a forest, there is no such $c$.
	\end{theorem}
	
	It is straightforward, using \cref{thm:rodl} and \cref{pure1} (applied in the complement), to deduce the following, a version of \cref{pure1} 
	without the sparsity hypothesis:
	\begin{theorem} \label{pure1mixed}
		If $H$ is a forest, then for all $d$ with $0<d\le 1/2$ there exists $c>0$ such that if $G$ is an $\overline{H}$-free graph  with $|G|\ge 2$, then there exist disjoint $A,B\subseteq V(G)$ with $|A|,|B|\ge c|G|$ such that either $A,B$ are complete, or $A,B$ are weakly $d$-sparse to each other. If $H$ is not a forest, then for all $d$ with $0<d\le 1/2$, there is no such $c$.
	\end{theorem}
	Thus if we exclude the complement of a forest\footnote{Note that we have switched from excluding a forest to excluding the complement of a forest: for convenience, we will work in the complement for most of this paper and exclude $\overline P_5$.} then we get a blockade of length 2 and linear width that is either complete or very sparse (as opposed to Theorem \ref{pure1a}, where the stronger assumption enables us to obtain a pair that is complete or anticomplete).

	Similarly, for longer blockades,	 
	it is helpful to move away from completeness or anticompleteness, and introduce a parameter $\eps>0$ that allows a certain amount of ``noise'' (parameterized by $\eps$) in part of the definition.  However, we insist on the following:
	\begin{itemize}
		\item the length and width of the blockade have a {\em polynomial} dependence on $\eps$ (or $1/\eps$);
		\item we can obtain such a blockade for {\em every} $\eps\in(0,1/2)$ (unlike the blockades we get from the
		quasi-Erd\H os-Hajnal property, which only need to exist for some $k$).
	\end{itemize}
	
	Let $H$ be a graph: we say that $H$ is {\em nice} (for lack of a better word) if there exist $a,b>0$ such that for every $\overline{H}$-free graph $G$ and every $\eps$ with $0<\eps\le 1/2$, there is an $(\eps^{-1},\floor{\eps^{a}\abs G})$-blockade $(B_1,\ldots,B_{\ell})$ in $G$, such that for all distinct $i,j\in[\ell]$, $(B_i,B_j)$ is either complete or weakly $\eps^b$-sparse in $G$. 
	
	A key lemma of this paper is that $P_5$ is nice; but before we 
	go on to its proof, let us consider niceness in general. Which graphs are nice? By taking $\eps=1/2$, \cref{pure1mixed} 
	implies that every nice graph is a forest; but perhaps all forests are nice. We have not been able to decide that, but
	we would like to make three points:
	\begin{itemize}
		\item Perhaps niceness is a halfway point towards proving 
		\cref{thm:main} for forests, because every forest $H$ with the polynomial R\"odl property is nice. To see this, suppose $H$ has the 
		polynomial R\"odl property; then we have some $d>0$ such that for every $\eps\in(0,\frac12)$ and every $\overline{H}$-free 
		graph $G$, there exists an $\eps^{2d}$-restricted $S\subseteq V(G)$ with $\abs S\ge \eps^{2d^2}\abs G$.
		If $G[S]$ is $\eps^{2d}$-sparse then it is easy to get a weakly $\eps^d$-sparse $(\eps^{-1},\floor{\eps^{10d^2}\abs G})$-blockade in $G[S]$ (by taking a suitable partition).
		If $\overline G[S]$ is $\eps^{2d}$-sparse then we can increase $d$ if necessary and iterate \cref{pure1} 
		to get a complete $(\eps^{-1},\floor{\eps^{10d^2}\abs G})$-blockade in $G[S]$ (we omit the details). 
		
		\item The niceness of a forest $H$ by itself does not seem  enough to prove the \erh{} conjecture (or polynomial R\"odl property) for $H$ 
		directly. Niceness gives us a blockade in which all the pairs are sparse or complete. We can make a graph with a vertex for each block, with an edge for each complete pair
		of blocks, and we would know that this ``pattern graph'' is $\overline{H}$-free, but we know nothing else about it. If we 
		apply induction to it, we prove just the ``near-polynomial R\"odl'' property of $H$ (that is, $\delta$ can be taken as $2^{-(\log\frac1\eps)^{1+o(1)}}$ in \cref{thm:rodl}), which implies the ``near-\erh{}'' property ($2^{(\log n)^{1-o(1)}}$ in place of $n^c$).
		
		\item Let us say $H$ is {\em strongly nice} if it satisfies the niceness condition with ``weakly $\eps^d$-sparse'' changed to ``anticomplete'': in other words, we require an $(\eps^{-1},\floor{\eps^{a}\abs G})$-blockade $(B_1,\ldots,B_{\ell})$ in $G$, such that for all distinct $i,j\in[\ell]$, $(B_i,B_j)$ is either complete or anticomplete. This is too strong to be interesting, because when $\eps$ is a constant that would mean every $H$-free graph contains a linear pure pair, which is not true unless $|H|\le 4$ (see~\cite{MR4170220}). 
		
		In the other direction, let us say $H$ is {\em weakly nice} if it satisfies the niceness condition with 
		``complete'' changed to ``weakly $\eps^b$-sparse in $\overline{G}$'': thus we ask for an $(\eps^{-1},\floor{\eps^{a}\abs G})$-blockade such that for all distinct $i,j\in[\ell]$, $(B_i,B_j)$ is weakly $\eps^b$-sparse in $G$ or $\overline G$. 
		This {\em is} still an interesting property. 
		We do not know that being weakly nice is equivalent to either the polynomial R\"odl property or the near-polynomial R\"odl property, but it is somewhere between them: every graph $H$ with the polynomial R\"odl property is weakly nice (not just forests); and every weakly nice graph has the near-polynomial R\"odl property.
		Also, it is not hard to see that if $H$ is weakly nice and satisfies the \erh{} conjecture then it has the polynomial R\"odl property.
	\end{itemize}
	
	Returning to $P_5$, the proof of \cref{thm:main} is in two parts: first we prove a lemma, and then we use the lemma to prove the main theorem. The lemma is of interest in its own right:
	\begin{lemma}
		\label{lem:robust0}
		There exists $d\ge40$ for which the following holds.
		Let $\eps\in(0,\frac12)$, and let $G$ be a $\overline{P_5}$-free graph with $\abs G\ge \eps^{-10d^2}$.
		Then there is an $(\eps^{-1},{\eps^{10d^2}\abs G})$-blockade $(B_1,\ldots,B_{\ell})$ in $G$ such that for all distinct $i,j\in[\ell]$, $(B_i,B_j)$ is either complete or  weakly $\eps^d$-sparse in $G$.
	\end{lemma}
	
	We prove \cref{lem:robust0} in \cref{sec:conductors}.  Then we apply it in \cref{sec:2ndround} to deduce:
	\begin{lemma}
		\label{lem:house0}
		There exists $a\ge1$ such that the following holds.
		For every $x\in(0,\frac12)$ and every $\overline{P_5}$-free graph $G$, either:
		\begin{itemize}
			\item $G$ has an $x$-restricted induced subgraph with at least $x^a\abs G$ vertices; or
			
			\item there is a complete or anticomplete $(k,\abs G/k^a)$-blockade in $G$, for some $k\in[2,1/x]$.
		\end{itemize}
	\end{lemma}
	The main result, \cref{thm:main}, can be deduced from \cref{lem:house0}
	in a page or so.
	
	Both the proof of \cref{lem:robust0}, and its application to prove \cref{lem:house0}, use a process of 
	iterative sparsification, which was introduced in \cite{density3,density4}
	and can be summarized as follows.
	We start with a graph $G$, that is $H$-free for some fixed $H$, and we are given $x$ with $0<x\le 1/2$. In order to prove
	the polynomial R\"odl property for $H$, we need to show that $G$ contains an $x$-restricted induced subgraph with at least 
	$\poly(x)\abs G$ vertices, where the polynomial depends on $H$ but not on $G$.  We can 
	assume that $x$ is at most any positive constant that is convenient. For the method to work, there needs to be a lemma that says 
	that for any value of $y\ge x$, if we have an induced subgraph $F$ of $G$ that is $y$-restricted, then either 
	\begin{itemize}
		\item there is an induced subgraph $F'$ of $F$ with $|F'|\ge \poly(y') |F|$ that is $y'$-restricted, where $y^D\le y'\le y^d$ 
		for some fixed $D\ge d>1$; or
		\item some other good thing happens.
	\end{itemize}
	
	To use the lemma, we choose a subgraph $F$ of linear size that is $y$-restricted for $y$ some small constant 
	(we can do this, for instance by applying R\"odl's theorem). Now we apply the lemma to $F$, and, if the ``other good thing'' does 
	not happen, we find $F'$ and $y'$. Repeat, and if the ``other good thing'' never happens, we
	recursively generate a nested sequence of induced subgraphs
	that are $y$-restricted for smaller and smaller values of $y$, and with size at least some polynomial in (the current value of) $y$
	times $|G|$. If $y$ becomes smaller than the target $x$, then the first time it does so, it is not {\em much} smaller than 
	$x$ (because it is not much smaller than the previous value of $y$), and then we have the $x$-restricted induced subgraph that we 
	wanted. So we can assume that at some stage the ``other good thing'' happens. 
	
	\section{Preliminaries}
	\label{sec:prelim}
	In this section we gather several basic results.
	A graph $G$ is {\em anticonnected} if $\overline G$ is connected; and an induced subgraph $F$ of $G$ is an {\em anticonnected component} of $G$ if $\overline F$ is a connected component of $\overline G$.
	The following fact says that graphs without large anticonnected components contain long and wide complete blockades.
	\begin{lemma}
		\label{lem:cores}
		Let $k\ge2$ be an integer, and let $G$ be a graph whose anticonnected components have size less than $\abs G/k$. Then there is a complete $(k,\abs G/k^2)$-blockade in $G$.
	\end{lemma}
	\begin{proof}
		By the hypothesis, there exists $n\ge0$ minimal for which there is a partition $S_0\cup S_1\cup\cdots\cup S_n=V(G)$
		such that $(S_0,S_1,\ldots,S_n)$ is a complete blockade in $G$ with $\abs{S_i}<\abs G/k$ for all $i\in\{0,\dots,n\}$.
		In particular $n+1>k$ and so $n\ge k$.
		We may assume $\abs{S_0}\le\abs{S_1}\le\ldots\le\abs{S_n}$.
		If there exists $i\ge1$ with $\abs{S_i}<\abs G/(2k)$,
		then $\abs{S_{i-1}\cup S_i}<\abs G/k$ and so $(S_0,\ldots,S_{i-2},S_{i-1}\cup S_i,S_{i+1},\ldots,S_n)$ would contradict the minimality of $n$.
		Hence $\abs{S_i}\ge\abs G/(2k)\ge\abs G/k^2$ for all $i\ge1$;
		and so $(S_1,\ldots,S_n)$ is a complete $(k,\abs G/k^2)$-blockade in $G$.
		This proves \cref{lem:cores}.
	\end{proof}
	The following simple probabilistic lemma will be useful in \cref{sec:1stround}.
	\begin{lemma}
		\label{lem:locate}
		Let $x\in(0,\frac12)$.
		Let $G$ be a bipartite graph with bipartition $(A,B)$ where every vertex in $B$ has at least $x\abs A$ neighbours in $A$.
		Then there exists $A'\subset A$ such that $\abs {A'}\le1/x$ and there are at least $\frac12\abs B$ vertices in $B$ with a neighbour in $A'$.
	\end{lemma}
	\begin{proof}
		Let $k:=\floor{1/x}$; we may assume that $|A|\ge k$. Choose $s_1\LL s_k\in A$ uniformly and independently at random, and let $S=\{s_1\LL s_k\}$.
		For each $v\in B$, since $v$ has at least $x\abs A$ neighbours in $A$, the probability that none of $s_1\LL s_k$ is such a neighbour is at most
		\[\left(\frac{\abs A-x\abs A}{\abs A}\right)^k
		=(1-x)^{\floor{1/x}}.\]
		If $x> 1/3$, then $(1-x)^{\floor{1/x}}= (1-x)^2\le 4/9\le 1/2$. If $x\le 1/3$, then $x\floor{1/x}\ge 3/4$, and so
		\[(1-x)^{\floor{1/x}} \le e^{-x\floor{1/x}} \le e^{-3/4}\le 1/2.\]
		So, in either case, the expected number of vertices in $B$ with no neighbour in $S$ is at most $|B|/2$;
		and hence there is a choice of $A'\subseteq A$ with the desired property.
		This proves \cref{lem:locate}.
	\end{proof}
	For $\ell,w\ge0$ and a graph $G$,
	an {\em $(\ell,w)$-comb} in $G$ is a sequence of pairs $((a_i,B_i):i\in[\ell])$ where
	\begin{itemize}
		\item $(B_1,\ldots,B_{\ell})$ is an $(\ell,w)$-blockade in $G$;
		
		\item $a_1,\ldots,a_\ell$ are pairwise distinct, and $\{a_1,\ldots,a_\ell\},B_1,\ldots,B_\ell$ are pairwise disjoint subsets of $V(G)$; and
		
		\item for all distinct $i,j\in[\ell]$, $a_i$ is adjacent to every vertex of $B_i$ in $G$ and nonadjacent to every vertex of $B_j$ in $G$.
	\end{itemize}
	We call $a_1\LL a_\ell$ the {\em apexes} of the comb.
	
	To prove \cref{lem:robust0}, we need a special case of the \dd comb\ee{} lemma from~\cite{MR4563865}. 
	\begin{lemma}
		\label{lem:combs}
		Let $G$ be a graph and let $A,B\subseteq V(G)$ be nonempty and disjoint, such that
		each vertex in $A$ has at most $\Delta>0$ neighbours in $B$.
		Then either:
		\begin{itemize}
			\item at most $20\sqrt{\abs B\Delta}$ vertices in $B$ have a neighbour in $A$; or
			
			\item
			for some integer $k\ge1$, there is a $(k,\abs B/k^2)$-comb $((a_i,B_i):i\in[k])$ in $G$ where $a_i\in A$ and $B_i\subset B$ for all $i\in[k]$.
		\end{itemize} 
	\end{lemma}
	The final ingredient we need is a well-known result for sparse $P_5$-free graphs~\cite{MR3343757}, a special case of \cref{pure1}.
	We include a short proof here for completeness.
	\begin{lemma}
		\label{lem:sparse}
		Let $\eta=2^{-5}$; then for
		every $\eta$-sparse $P_5$-free graph $G$ with $\abs G\ge2$, there is an anticomplete $(2,{\eta\abs G})$-blockade in $G$.
	\end{lemma}
	\begin{proof}
		Let $G$ be $\eta$-sparse and $P_5$-free with $\abs G\ge2$; and suppose that there is no anticomplete $(2,{\eta\abs G})$-blockade in $G$.
		If $\abs G<\eta^{-1}$ then $G$ is edgeless (since it is $\eta$-sparse) and we are done by taking the endpoints of an arbitrary nonedge of $G$.
		Thus we may assume $\abs G\ge\eta^{-1}$.
		
		Now, by \cref{lem:cores} applied to $\overline G$ with $k=2$, $G$ has a connected component $F$ with $\abs F\ge\frac12\abs G$.
		Let $v\in V(F)$ and $A$ be the set of neighbours of $v$ in $F$; then $A\ne\emptyset$.
		Let $F':=F\setminus(A\cup\{v\})$.
		Since $G$ is $\eta$-sparse, we have $\abs{F'}\ge \abs F-\abs A-1\ge (\frac12-2\eta)\abs G\ge\frac13\abs G$; and therefore $\frac14\abs{F'}\ge \eta\abs G$,
		\cref{lem:cores} gives a connected component $J$ of $F'$ with $\abs J\ge\frac12\abs{F'}\ge\frac16\abs G$.
		Since $F$ is connected, there exists $u\in A$ with a neighbour if $V(J)$.
		Let $B$ be the set of vertices in $J$ adjacent to $u$ in $F$;
		then $1\le\abs B\le \eta\abs G$.
		Thus $\abs{J\setminus B}\ge \frac16\abs G-\eta\abs G\ge\frac18\abs G=4\eta\abs G$.
		Again, by \cref{lem:cores} with $k=2$,
		$J\setminus B$ has a connected component $J'$ with $\abs{J'}\ge\frac12\abs{J\setminus B}\ge 2\eta\abs G$.
		Since $J$ is connected, there exists $w\in B$ with a neighbour in $V(J')$.
		Because $w$ has degree at most $\eta\abs G<\abs{J'}$, $w$ also has a nonneighbour in $V(J')$.
		Hence the connectedness of $J'$ gives $z,z'\in V(J')$ with $wz\in E(J),wz'\nin E(J)$; but then $\{u,v,w,z,z'\}$ forms a copy of $P_5$ in $G$, a contradiction.
		This proves \cref{lem:sparse}.
	\end{proof}
	\section{Using a comb}
	\label{sec:1stround}
	We will obtain \cref{lem:robust0} as a consequence of the following:
	\begin{lemma}
		\label{lem:house40}
		There exists $d\ge40$ for which the following holds.
		For every $x\in(0,2^{-d})$ and every $\overline{P_5}$-free graph $G$ with $\abs G\ge x^{-d}$, there exist $k\in[2,1/x]$ and a pure or $x$-sparse $(k,\abs G/k^d)$-blockade in $G$.
	\end{lemma}
	A crucial step of the proof of \cref{lem:house40} is the following lemma:
	
	\begin{lemma}
		\label{lem:house1}
		Let $x,y>0$ with $x\le y\le 2^{-8}$,
		and let $G$ be a $y^3$-sparse $\overline{P_5}$-free graph with $\abs G\ge y^{-4}$. Then either:
		\begin{itemize}
			\item $G$ is $2y^4$-sparse;
			
			\item there exist $k\in[y^{-1/4},1/x]$ and a pure $(k,\abs G/k^{26})$-blockade in $G$; or
			
			\item there are disjoint $X,Y\subset V(G)$ such that $\abs X\ge{y^4\abs G}$, $\abs Y\ge(1-4y)\abs G$, and $Y$ is $x$-sparse to~$X$.
		\end{itemize}
	\end{lemma}
	
	Let us sketch the proof of this result.
	We are given sufficiently small positive variables $x\le y$ and a $y$-sparse $\overline{P_5}$-free graph $G$.
	If $G$ is actually $y/2$-sparse then $G$ is already (much) sparser than what we knew about it and so the first outcome of \cref{lem:house1} holds.
	Thus, let us assume that there is a vertex $v$ of degree at least $(y/2)\abs G$ in $G$.
	Our plan now is either to extract a pure blockade with appropriate length and width from the neighbourhood of $v$ (the second outcome of \cref{lem:house1}), or to conclude that a significant portion of its nonneighbourhood is $x$-sparse to a decent portion of its neighbourhood (the third outcome of \cref{lem:house1}).
	
	To carry this out, we will first apply \cref{lem:combs} to obtain a comb between the neighbourhood $B$ of $v$ and the rest of the graph;
	but instead of taking a comb with apexes in $B$ that expands into the rest of $G$ 
	(as was done in~\cite{MR4563865}),
	we will build an \dd upside-down\ee{} comb $((a_i,B_i):i\in[k])$ (for some $k\ge1$), with apexes in $V(G)\setminus B$
	that goes from the rest of $G$ back into $B$ (in other words, $v$ is nonadjacent to $a_1,\ldots,a_k$ and adjacent to every vertex in $B_1\cup\cdots\cup B_k$; see \cref{fig:comb}).
	Such a comb is potentially useful, because if we can arrange for every $G[B_i]$ to be anticonnected (\cref{lem:cores}), then the blockade $\mac B=(B_1,\ldots,B_k)$ has to be pure: whenever there is a vertex from some $B_j$ mixed on another block $B_i$, the anticonnectivity of $G[B_i]$ would then give a copy of the house $\overline{P_5}$ in $G$ that contains $v$ and $a_i$ (see \cref{fig:comb}).

	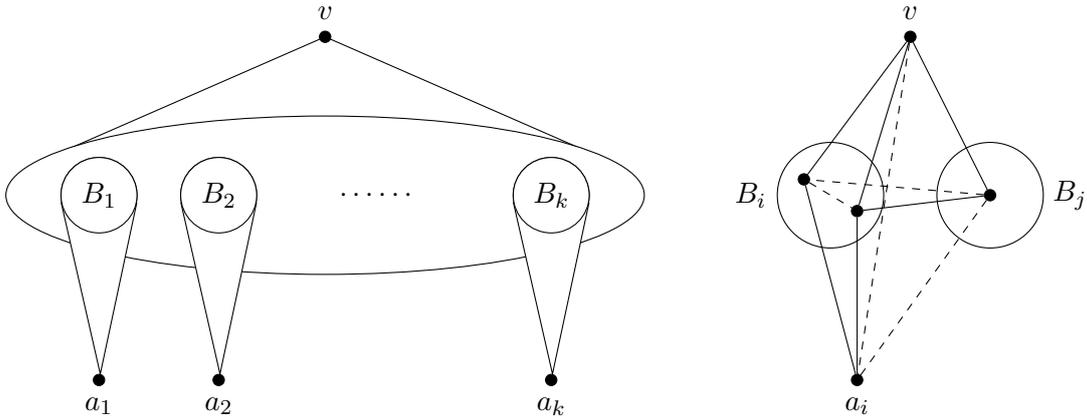
\begin{figure}[ht]
		\centering
		
		\begin{tikzpicture}[scale=0.7,auto=left]
			\filldraw[fill=white]
			(0,3) -- (-5.45,0.6) -- (5.45,0.6) -- cycle;
			\draw [fill=white] (0,0) ellipse (6cm and 1.5cm);
			\node[inner sep=1.5pt, fill=black,circle,draw] (v) at ({0},{3}) {};
			\node[above=0.1cm] at (v) {$v$};
			
			\node[] (b1) at (-4.25,0) {$B_1$};
			\node[inner sep=1.5pt, fill=black,circle,draw] (a1) at ({-4.25},{-3.5}) {};
			\node [circle,draw,name path=circle,fill=white] (c1) at (b1) [minimum size=1cm] {};
			\filldraw[fill=white] (a1)  
			-- (tangent cs:node=c1,point={(a1)},solution=1) 
			coordinate[] () 
			-- (tangent cs:node=c1,point={(a1)},solution=2)
			coordinate[] () 
			-- cycle;
			\node [circle,draw,name path=circle,fill=white] (c1) at (b1) [minimum size=1cm] {};
			\node[] (b1) at (-4.25,0) {$B_1$};
			\node[below=0.1cm] at (a1) {$a_1$};
			
			
			\node[] (b2) at (-2,0) {$B_2$};
			\node[inner sep=1.5pt, fill=black,circle,draw] (a2) at ({-2},{-3.5}) {};
			\node [circle,draw,name path=circle,fill=white] (c2) at (b2) [minimum size=1cm] {};
			\filldraw[fill=white] (a2)  
			-- (tangent cs:node=c2,point={(a2)},solution=1) 
			coordinate[] () 
			-- (tangent cs:node=c2,point={(a2)},solution=2)
			coordinate[] () 
			-- cycle;
			\node [circle,draw,name path=circle,fill=white] () at (b2) [minimum size=1cm] {};
			\node[] (b2) at (-2,0) {$B_2$};
			\node[below=0.1cm] at (a2) {$a_2$};
			
			\node[] (bk) at (4.25,0) {};
			\node[inner sep=1.5pt, fill=black,circle,draw] (ak) at ({4.25},{-3.5}) {};
			\node [circle,draw,name path=circle,fill=white] (ck) at (bk) [minimum size=1cm] {};
			\filldraw[fill=white] (ak)  
			-- (tangent cs:node=ck,point={(ak)},solution=1) 
			coordinate[] () 
			-- (tangent cs:node=ck,point={(ak)},solution=2)
			coordinate[] () 
			-- cycle;
			\node [circle,draw,name path=circle,fill=white] (ck) at (bk) [minimum size=1cm] {};
			\node[] at (bk) {$B_k$};
			\node[below=0.1cm] at (ak) {$a_k$};
			
			\node[] at (1,0) {$\cdots\cdots$};
			
			\draw [fill=white] (9.5,0) circle (1cm);
			\node[inner sep=1.5pt, fill=black,circle,draw] (a1) at ({10},{-3.5}) {};
			
			\draw [fill=white] (12.5,0) circle (1cm);
			
			\node[below=0.1cm] at (a1) {$a_i$};
			\node[] at (8,-0) {$B_i$};
			\node[] at (14,0) {$B_j$};
			
			\node[inner sep=1.5pt, fill=black,circle,draw] (v) at ({11},{3}) {};
			\node[above=0.1cm] at (v) {$v$};
			
			\node[inner sep=1.5pt, fill=black,circle,draw] (v1) at ({9},{0.3}) {};
			\node[inner sep=1.5pt, fill=black,circle,draw] (v2) at ({10},{-0.3}) {};
			\node[inner sep=1.5pt, fill=black,circle,draw] (v3) at ({12.5},{0}) {};
			\draw[-] (v) -- (v3) -- (v2) -- (a1) -- (v1) -- (v) -- (v2);
			\draw[dashed] (v) -- (a1) -- (v3) -- (v1) --  (v2);
		\end{tikzpicture}
		
		\caption{Making a house from an upside-down comb with anticonnected blocks.}
		\label{fig:comb}
	\end{figure}
	
	Thus, $\mac B$ is pure; but to satisfy the lemma, it must have the right length and width. First, 
	we need its width to be at least $\poly(1/k)|G|$ where $k$ is its length. 
	The blocks $B_1,\ldots,B_k$ are subsets of $B$; and the application of \cref{lem:combs} tells us that $\mac B$ is a 
	$(k,\abs B/O(k^2))$-blockade in $G[B]$,
	and so a $(k,(y/2)\abs G/O(k^2))$-blockade in $G$, but it gives us no lower bound on $k$.
	To ensure that the width of $\mac B$ is at least $\poly(1/k)|G|$, we need $k$ to be at least some small power of 
	$y^{-1}$. But we can arrange this as follows. 
	Let us choose the comb so that it contains no vertices outside $B$ that see at least a $y^{1/2}$ fraction of $B$. 
	There are not many such vertices (at most $O(y^{1/2})\abs G$), because $\abs B\ge y\abs G$ and everyone in $B$ sees at most $y\abs G$ vertices outside.
	In other words, by letting $A$ be the set of vertices with at most $y^{1/2}\abs B$ neighbours in $B$, we have $\abs A\ge (1-O(y^{1/2}))\abs G$;
	so let us choose the comb with every apex $a_i$ in $A$. Then the width of the comb is at least $\abs B/O(k^2)$ and at most
	$y^{1/2}\abs B$, and this ensures that $k\ge \Omega(y^{-1/4})$,  as we wanted. 
	Consequently we can arrange that $\mac B$ is a pure $(k,\abs G/O(k^{6}))$-comb in $G$.
	
	Another thing we need, for $\mac B$ to satisfy the lemma, is a good {\em upper} bound on its length $k$.
	We can arrange that
	$k\le \poly(1/x)$ (or another good thing happens), by putting a further restriction on how we choose the comb.
	Indeed, given $A,B$ as above, if there are too many vertices of $B$ (at least half, say) seeing fewer than $x^2\abs A$ vertices 
	in $A$, then it is easy to obtain subsets $A'\subseteq A$, $B'\subseteq B$ with $\abs {A'}\ge(1-O(x))\abs A\ge(1-O(y^{1/2}))\abs G$ and $\abs{B'}\ge\frac12\abs B\ge \Omega(y)\abs G$ such that $A'$ is $x$-sparse to $B'$;
	and this satisfies the third outcome of \cref{lem:house1}.
	So we may assume there are at least $\frac12\abs B$ vertices of $B$ with at least $x^2\abs A$ neighbours in $A$;
	and then \cref{lem:locate} gives us some subset $S$ of $A$ of size at most $x^{-2}$ that \dd covers\ee{} a constant fraction of $B$.
	By \cref{lem:combs}, the apexes $a_1,\ldots,a_k$ of the comb can be taken from $S$, and so $k\le x^{-2}$ as a consequence.
	
	That was a sketch of the proof of \cref{lem:house1}.
	Next we will write it out, with cosmetic adjustments in the constant factors and exponents.
	\begin{proof}
		[Proof of \cref{lem:house1}]
		Assume that the first and third outcomes do not hold.
		Since the first outcome does not hold, $G$ has a vertex $v$ of degree at least $2y^4\abs G$. Let $N$ be its set of neighbours.
		
		\begin{claim}
			\label{claim:house1}
			There exist $A\subseteq V(G)\setminus (N\cup \{v\})$ and $B\subseteq N$ such that
			\begin{itemize}
				\item $\abs B\ge y^4\abs G$ and $\abs A\ge (1-3y)\abs G$; and
				
				\item $A$ is $y^2$-sparse to $B$ and every vertex in $B$ has at least $x^2\abs A$ neighbours in $A$.
			\end{itemize}
		\end{claim}
		
		\begin{subproof}
			We have $\abs N\ge 2y^4\abs G$.
			Let $A'$ be the set of vertices in $V(G)\setminus(N\cup\{v\})$ with at least $\frac12y^2\abs N$ neighbours in $N$.
			By averaging, there is a vertex in $N$ with at least $\frac12y^2\abs{A'}$ neighbours in ${A'}$;
			and so $\frac12y^2\abs{A'}\le y^3\abs G$ since $G$ is $y^3$-sparse,  which yields that $\abs{A'}\le 2y\abs G$.
			Let $A:=V(G)\setminus(N\cup A'\cup\{v\})$; then since $1+y^2\abs G\le y\abs G$, we have
			\[\abs A\ge \abs G-(1+y^3\abs G+2y\abs G)
			\ge (1-3y)\abs G.\]
			Let $N'$ be the set of vertices in $N$ with at most $x^2\abs A$ neighbours in $A$, and let $B:=N\setminus N'$.
			There are at most $x\abs A$ vertices in $A$ with more than $x\abs{N'}$ neighbours in $N'$, since there are at most $x^2|A|\cdot |N'|$ edges between $A$ and $N'$; so there are at least
			\[\abs A-x\abs A\ge (1-3y)\abs G-x\abs G\ge (1-4y)\abs G\]
			vertices in $A$ with at most $x\abs{N'}$ neighbours in $N'$.
			Thus, $\abs {N'}\le y^4\abs G\le \frac12\abs N$, since the third outcome of the lemma does not hold, and so
			\[\abs B=\abs N-\abs{N'}\ge \textstyle\frac12\abs N\ge y^4\abs G.\]
			Since $A$ is $\frac12y^2$-sparse to $N$, it is $y^2$-sparse to $B$.
			This proves \cref{claim:house1}.
		\end{subproof}
		
		Let $A,B$ be given by \cref{claim:house1};
		then \cref{lem:locate} (with $x^2$ in place of $x$) gives $S\subset A$ with $\abs S\le x^{-2}$ such that there are 
		at least $\frac12\abs B$ vertices in $B$ with a neighbour in $S$.
		Let $\Delta:=y^2\abs B$.
		Since $y<\frac1{40}$, more than 
		$20\sqrt{\abs B\Delta}=20y\abs B$ vertices in $B$ have a neighbour in $S$. So by 
		\cref{lem:combs},
		for some integer $\ell\ge1$, there is an $(\ell,\abs B/\ell^2)$-comb $((a_i,B_i):i\in[\ell])$ in $G$ where $a_i\in S$ and $B_i\subset B$ for all $i\in[\ell]$.
		
		Since $A$ is $y^2$-sparse to $B$, $\abs B/\ell^2\le y^2\abs B$ and so $\ell\in [y^{-1},x^{-2}]$.
		Let $k:=\ceil{\ell^{1/4}}\in[y^{-1/4},1/x]$; then $\abs B\ge y^4\abs G\ge \abs G/\ell^4
		\ge\abs G/k^{16}$ and $(B_1,\ldots,B_{k})$ is a $(k,\abs B/k^{8})$-blockade
		(note that $k\le\sqrt\ell\le x^{-1/2}$).
		Let $I:=[k]$.
		\begin{claim}
			\label{claim:pure}
			There is a pure $(k,\abs B/k^{10})$-blockade in $G[B]$.
		\end{claim}
		\begin{subproof}
			For each $i\in I$, if $G[B_i]$ has no anticonnected component of size at least $\abs{B_i}/k$, then \cref{lem:cores} gives a complete $( k,\abs{B_i}/k^2)$-blockade in $G[B_i]$ (note that $k\ge y^{-1/4}\ge 4$);
			and this satisfies the claim since $\abs{B_i}/k^2\ge \abs B/k^{10}$.
			Hence, we may assume each $G[B_i]$ has an anticonnected component $D_i$ with 
			\[\abs{D_i}\ge\abs{B_i}/k^2\ge\abs B/k^{10}.\]
			
			For distinct $i,j\in I$, if there exists some $u\in D_j$ mixed on $D_i$, then $u$ would have a neighbour $w\in D_i$ and a nonneighbour $z\in D_i$ such that $wz\nin E(G)$ since $D_i$ is anticonnected; and so $\{v,u,w,z,a_i\}$ would form a copy of $\overline{P_5}$ in $G$ (see \cref{fig:comb}), a contradiction.
			Thus $(D_i:i\in I)$ is a pure blockade in $G[B]$ of length $k$ and width at least $\abs B/k^{10}$.
			This proves \cref{claim:pure}.
		\end{subproof}
		
		Since $\abs B/k^{10}\ge \abs G/k^{26}$, \cref{claim:pure} gives a pure $(k,\abs G/k^{26})$-blockade in $G$, which is the second outcome of the lemma.
		This proves \cref{lem:house1}.
	\end{proof}
	The rest of this section deals with the proof of \cref{lem:house40}.
	We first iterate \cref{lem:house1} to turn its third outcome (an $x$-sparse pair) into an $x$-sparse blockade outcome, as follows.
	\begin{lemma}
		\label{lem:house2}
		Let $c:=2^{-8}$.
		Let $x,y>0$ with $x\le y\le c$,
		and let $G$ be a $cy^3$-sparse $\overline{P_5}$-free graph with $\abs G\ge y^{-6}$. Then~either:
		\begin{itemize}
			\item there exists $S\subset V(G)$ such that $\abs S\ge c\abs G$ and $G[S]$ is $2y^4$-sparse;
			
			\item there exist $k\in[y^{-1/4},1/x]$ and a pure $(k,\abs G/k^{30})$-blockade in $G$; or
			
			\item there is an $x$-sparse $(y^{-1},{y^6\abs G})$-blockade in $G$.
		\end{itemize}
	\end{lemma}
	\begin{proof}
		Suppose that none of the outcomes holds.
		Thus there exists $n\ge0$ maximal such that there is an $x$-sparse blockade $(B_0,B_1,\ldots,B_n)$ with $\abs{B_{i-1}}\ge{y^6\abs G}$ for all $i\in[n]$ and $\abs{B_n}\ge(1-4y)^n\abs G$.
		Since the third outcome does not hold, $n<y^{-1}$; and so by the inequality $1-t\ge4^{-t}$ for all $t\in[0,\frac12]$, we have
		\[\abs{B_n}\ge (1-4y)^n\abs G
		\ge 4^{-4yn}\abs G>4^{-4}\abs G=c\abs G
		\ge y\abs G\ge x\abs G.\]
		Hence $G[B_n]$ has maximum degree at most $cy^3\abs G<y^3\abs{B_n}$; and since the first outcome does not hold,
		$G[B_n]$ is not $2y^4$-sparse.
		Therefore, by \cref{lem:house1}, either:
		\begin{itemize}
			\item there exist $k\in[y^{-1/4},1/x]$ and a pure $(k,\abs{B_n}/k^{26})$-blockade in $G$; or
			
			\item there are disjoint $X,Y\subset B_n$ such that $\abs X\ge{y^5\abs{B_n}}$, $\abs Y\ge(1-4y)\abs{B_n}$, and $Y$ is $x$-sparse to $X$.
		\end{itemize}
		
		The first bullet cannot hold since $\abs{B_n}/k^{26}\ge y\abs G/k^{26}\ge \abs G/k^{30}$ and the second outcome of the lemma does not hold.
		Thus the second bullet holds; but then $(B_0,B_1,\ldots,B_{n-1},X,Y)$ would contradict the maximality of $n$ since $\abs X\ge{y^5\abs{B_n}}\ge{y^6\abs G}$.
		This proves \cref{lem:house2}.
	\end{proof}
	The next result contains the ``iterative sparsification'' step of the proof. It allows us to replace
	the $cy^3$-sparsity hypothesis of \cref{lem:house2} with a ``sparsity a small constant'' hypothesis and still
	deduce (essentially) the same conclusion.
	\begin{lemma}
		\label{lem:house3}
		Let $c:=2^{-8}$. Let $x\in(0,c^5)$, and let $G$ be a $c^{16}$-sparse $\overline{P_5}$-free graph with $\abs G\ge x^{-7}$.
		Then either:
		\begin{itemize}
			\item for some $k\in[1/c,1/x]$, there is a pure $(k,\abs G/k^{34})$-blockade in $G$; or
			
			\item for some $y\in[x,c^5]$, there is an $x$-sparse $(y^{-1},{y^7\abs G})$-blockade in $G$.
		\end{itemize}
	\end{lemma}
	\begin{proof}
		Suppose that neither of the two outcomes holds.
		Let $y\in[cx,c^5]$ be minimal such that $G$ has a $cy^3$-sparse induced subgraph $F$ with $\abs F\ge y\abs G$. (This is possible, since taking $y=c^5$ has the property.)
		Suppose that $y<x$; then $F$ is $x^3$-sparse with $\abs F\ge y\abs G\ge cx\abs G\ge x^2\abs G\ge x^{-5}$.
		Because $\ceil{x^{-1}}\cdot\ceil{\frac14x\abs F}\le 2x^{-1}\cdot \frac12x\abs F=\abs F$,
		there is an $(x^{-1},\frac14x\abs F)$-blockade in $F$,
		which is then $x$-sparse since $\frac14x\ge x^2$.
		Thus, since $\frac14x\abs F\ge \frac14 cx^2\abs G\ge x^3\abs G$, this would be an $x$-sparse $(x^{-1},x^3\abs G)$-blockade in $G$, a contradiction.
		
		Consequently $y\ge x$. 
		By \cref{lem:house2} applied to $F$, either:
		\begin{itemize}
			\item $F$ has a $2y^4$-sparse induced subgraph with at least $c\abs F\ge cy\abs G$ vertices;
			
			\item there exist $k\in[y^{-1/4},1/x]\subset[1/c,1/x]$ and a pure $(k,\abs F/k^{30})$-blockade in $F$; or
			
			\item there is an $x$-sparse $(y^{-1},{y^6\abs F})$-blockade in $F$.
		\end{itemize}
		The first bullet would give a $2y^4$-sparse induced subgraph of $F$ (and so of $G$) with at least $cy\abs G$ vertices,
		which contradicts the minimality of $y$ since $2y^4\le c^4y^3= c(cy)^3$.
		If the second bullet holds, then since $\abs F/k^{30}\ge y\abs G/k^{30}\ge \abs G/k^{34}$, there would be a pure $(k,\abs G/k^{34})$-blockade in $G$, a contradiction.
		If the third bullet holds, then since $y^6\abs F\ge y^7\abs G$,
		there would be an $x$-sparse $(y^{-1},{y^7\abs G})$-blockade in $G$, a contradiction.
		This proves \cref{lem:house4}.
	\end{proof}
	Next, by applying  R\"odl's \cref{thm:rodl}, we remove the sparsity hypothesis in \cref{lem:house3} completely, and prove \cref{lem:house40}, which we restate:
	\begin{lemma}
		\label{lem:house4}
		There exists $d\ge40$ for which the following holds.
		For every $x\in(0,2^{-d})$ and every $\overline{P_5}$-free graph $G$ with $\abs G\ge x^{-d}$, there exist $k\in[2,1/x]$ and a pure or $x$-sparse $(k,\abs G/k^d)$-blockade in $G$.
	\end{lemma}
	\begin{proof}
		Let $c:=2^{-8}$, and $\eta:=2^{-5}$, and let $\xi:=c^{16}$.
		By \cref{thm:rodl}, there exists $\theta\in(0,1)$ such that every $\overline{P_5}$-free graph $G$ contains a $\xi$-restricted induced subgraph with at least $\theta\abs G$ vertices.
		We shall prove that every $d\ge40$ with $2^{d}\ge(\eta\theta)^{-1}$ satisfies the lemma.
		To show this, let $x\in(0,2^{-d})$, and let $G$ be $\overline{P_5}$-free with $\abs G\ge x^{-d}\ge\eta^{-1}$.
		We must show that there exists $k\in[2,1/x]$ such that there is a pure or $x$-sparse $(k,\abs G/k^d)$-blockade in $G$.
		By the choice of $\theta$, $G$ has a $\xi$-restricted induced subgraph $F$ with $\abs F\ge\theta\abs G$.
		If $\overline F$ is $\xi$-sparse, then since $\overline F$ is $P_5$-free, \cref{lem:sparse} gives an anticomplete $(2,{\eta\abs F})$-blockade in $\overline F$;
		and we are done since ${\eta\abs F}\ge {\eta\theta\abs G}\ge {2^{-d}\abs G}$ by the choice of $d$.
		Hence, we may assume that $F$ is $\xi$-sparse (and so is $c^{16}$-sparse).
		Since $x\in(0,2^{-d})\subset(0,c^5)$, \cref{lem:house3} implies that either:
		\begin{itemize}
			\item for some $k\in[1/c,1/x]$, there is a pure $(k,\abs S/k^{34})$-blockade in $F$; or
			
			\item for some $y\in[x,c^5]$, there is an $x$-sparse $(y^{-1},{y^7\abs F})$-blockade in $F$.
		\end{itemize}
		
		If the first bullet holds, then $\abs G\ge x^{-d}\ge k^d$, $k\ge 1/c=2^8$, and $d\ge 40$ which together imply
		\[\abs F/k^{34}\ge \theta\abs F/k^{34}\ge 2^{-d}\abs F/k^{34}\ge k^{-d/8}\abs G/k^{34}
		\ge\abs F/k^d;\]
		and so there would be a pure $(k,\abs F/k^d)$-blockade in $G$ and we are done.
		If the second bullet holds, then since
		\[{y^7\abs F}\ge {\theta y^7\abs G}\ge {2^{-d}y^7\abs G}\ge {y^{d/8+7}\abs G}
		\ge {y^d\abs G},\]
		there would be an $x$-sparse $(y^{-1},y^d\abs G)$-blockade in $G$ and we are done.
		This completes the proof of \cref{lem:house4}.
	\end{proof}
	\section{The proof of \cref{lem:robust0}} \label{sec:conductors}
	Next we will deduce \cref{lem:robust0} from \cref{lem:house4}. If we take $x$ to be a power of $\eps^d$, then \cref{lem:house4} 
	already gives us something like what we want for \cref{lem:robust0}, but the blockade we obtain might have length too small.
	If so, then it still has very large blocks, and we can apply \cref{lem:house4} to each block to get
	a longer blockade, and repeat. 
	This idea is formalized in the following general theorem (with no $\overline{P_5}$-free condition),  which is 
	a slight modification of a theorem of~\cite{density4}.
	\begin{theorem}
		\label{thm:key}
		Let $\eps\in(0,\frac12)$ and $d\ge1$,
		and let $G$ be a graph with $\abs G\ge \eps^{-10d^2}$.
		Let $x:=\eps^{5d}$.
		Assume that for every induced subgraph $F$ of $G$ with $\abs F\ge \eps^{d}\abs G$, there exists $k\in[2,1/x]$ such that there is a pure or $x$-sparse $(k,\abs F/k^d)$-blockade in $F$.
		Then there is an $(\eps^{-1},{x^{2d}\abs G})$-blockade $(B_1,\ldots,B_{\ell})$ in $G$,
		such that for all distinct $i,j\in[\ell]$, $(B_i,B_j)$ is either complete or weakly $\eps^d$-sparse in $G$.
	\end{theorem}
	\begin{proof}
		Let $J$ be a graph; and for each $j\in V(J)$ let $A_j$ be a nonempty subset of $V(G)$, pairwise disjoint, such that for all distinct $i,j\in J$, $A_i$ is complete to $A_j$ whenever $i,j$ are adjacent in $J$. We call $\mathcal{L}=(J,(A_j:j\in V(J)))$
		a {\em layout}. A pair $\{u,v\}$ of distinct vertices of $G$ is {\em undecided} for a layout $(J,(A_j:j\in V(J)))$
		if there exists $j\in V(J)$ with $u,v\in A_j$; and {\em decided} otherwise.
		A decided pair $\{u,v\}$ is {\em wrong} for $(J,(A_j:j\in V(J)))$ if there are distinct
		$i,j\in V(J)$ such that $u\in A_i$, $v\in A_j$, and $u,v$ are adjacent in $G$ while $i,j$ are nonadjacent in $J$.
		We are interested in layouts in which the number of wrong pairs is only a small fraction of the number of decided pairs.
		Choose a layout $\mathcal{L}=(J,(A_j:j\in V(J)))$ satisfying the following:
		\begin{itemize}
			\item $\abs{A_j}\ge \eps^{2d}\abs G$ for each $j\in V(J)$;
			\item $\sum_{j\in V(J)}|A_j|^{1/d}\ge |G|^{1/d}$;
			\item the number of wrong pairs is at most $x$ times the number of decided pairs; and
			\item subject to these three conditions, $|J|$ is maximum.
		\end{itemize}
		(This is possible since we may take $V(J)=\{1\}$ and $A_1=V(G)$ to satisfy the first three conditions.)
		\begin{claim}
			\label{claim:short}
			We may assume that $|J|\le \eps^{-1}$.
		\end{claim}
		\begin{subproof}
			Assume that $|J|\ge \eps^{-1}$. 
			Since the number of wrong pairs is at most $x$ times the number of decided pairs and so at most $x\abs G^2$,
			for every distinct $i,j\in V(J)$ that are nonadjacent in $J$,
			the number of edges between $A_i,A_j$ is at most $x\abs G^2\le x\eps^{-4d}\abs{A_i}\abs{A_j}=\eps^d\abs{A_i}\abs{A_j}$;
			that is, $(A_i,A_j)$ is weakly $\eps^d$-sparse.
			Since $\abs{A_i}\ge\eps^{2d}\abs G\ge x^{2d}\abs G$ for each $j\in V(J)$,
			$(A_j:j\in V(J))$ is thus a blockade satisfying the theorem.
			This proves \cref{claim:short}. 
		\end{subproof}
		
		Let $A\in\{A_j:j\in V(J)\}$ satisfy $|A|=\max_{j\in V(J)} |A_j|$. Since $\sum_{j\in V(J)}|A_j|^{1/d}\ge |G|^{1/d}$,
		and $|J|\le \eps^{-1}$ by \cref{claim:short}, it follows that
		$|A|^{1/d}\ge \eps|G|^{1/d}$, that is, 
		$|A|\ge \eps^{d}|G|$.
		By applying the hypothesis to
		$G[A]$,
		we obtain a pure or $x$-sparse $(k,\abs{A}/k^d)$-blockade $(B_1,\ldots, B_{k})$ in $G[A]$, for some $k\in[2,1/x]$.
		Let $K$ be the graph with vertex set $[k]$, such that for all distinct $p,q\in[k]$, $p$ is adjacent to $q$ in $K$ if and only if $B_p$ is complete to $B_q$ in $G[A]$; in particular $K$ is edgeless if $(B_1,\ldots,B_{k})$ is $x$-sparse in $G[A]$.
		\begin{claim}
			\label{claim:long}
			$k\ge \eps^{-1}$.
		\end{claim}
		\begin{subproof}
			Suppose that $k\le \eps^{-1}$. Then each of the sets $B_1,\ldots, B_{k}$ has size at least
			$|A|/k^d\ge \eps^{d} |A|$.
			By substituting $K$ for
			the vertex of $J$ corresponding to $A$, and replacing $A$ by $B_1,\ldots, B_\ell$, we obtain a new layout $\mathcal{L}'=(J', (A_j':j\in V(J')))$ say,
			where $|J'|>|J|$.
			We claim that this violates the choice of $\mathcal{L}$; and so we must verify that $\mathcal{L}'$ satisfies the
			first three bullets in the definition of $\mathcal{L}$. 
			To see this, observe that each $B_p$ satisfies
			$|B_p|\ge \eps^{d} |A|\ge \eps^{2d} |G|$,
			and so the first bullet is satisfied.
			For the second bullet, since $B_1,\ldots, B_{k}$ all have size at least $|A|/k^d$, it follows that
			$$|B_1|^{1/d}+\cdots+|B_{k}|^{1/d}\ge |A|^{1/d},$$
			and so $\sum_{j\in V(J')}|A_j'|^{1/d}\ge |G|^{1/d}$.
			For the third bullet, let
			$P$ be the set of all decided pairs for $\mathcal{L}$, and $Q\subseteq P$ the set of wrong pairs for $\mathcal{L}$;
			and define $P',Q'$ similarly for $\mathcal{L}'$. Then $P\subset P'$ and $|Q|\le x|P|\le x\abs{P'}$. Let $R$ be the
			set of all pairs $\{u,v\}$ with $u,v\in A$ such that $u,v$ belong to different blocks
			of $(B_1,\ldots, B_k)$. Then $R\subseteq  P'\setminus P$ and $Q'\setminus Q\subseteq R$.
			If $(B_1,\ldots,B_k)$ is pure in $G[A]$ then $\abs{Q'}\le\abs Q\le x\abs{P'}$;
			and if $(B_1,\ldots, B_k)$ is $x$-sparse in $G[A]$, then $\abs{Q'\setminus Q}\le x\abs R$ which yields
			$|Q'\setminus Q|\le x|P'\setminus P|$, and so
			$$|Q'|\le |Q|+|Q'\setminus Q|\le x|P|+ x|P'\setminus P|= x|P'|.$$
			This contradicts the choice of $\mathcal{L}$, and so proves \cref{claim:long}.
		\end{subproof}
		Since $k\le1/x$ and $\abs A\ge\eps^d\abs G\ge x^d\abs G$, we have $\abs{B_p}\ge \abs A/k^d\ge x^d\abs A\ge x^{2d}\abs G$ for each $p\in[k]$;
		and for all distinct $p,q\in[k]$,
		$(B_p,B_q)$ is either complete or weakly $\eps^d$-sparse since $x=\eps^{5d}\le \eps^d$.
		Hence $(B_1,\ldots,B_k)$ satisfies the theorem.
		This proves \cref{thm:key}.
	\end{proof}
	By combining \cref{lem:house4} and \cref{thm:key}, we prove \cref{lem:robust0}, which we restate:
	\begin{lemma}
		\label{lem:epsone}
		There exists $d\ge40$ for which the following holds.
		Let $\eps\in(0,\frac12)$, and let $G$ be a $\overline{P_5}$-free graph with $\abs G\ge \eps^{-10d^2}$.
		Then there is an $(\eps^{-1},{\eps^{10d^2}\abs G})$-blockade $(B_1,\ldots,B_{\ell})$ in $G$, such that for all distinct $i,j\in[\ell]$, $(B_i,B_j)$ is either complete or weakly $\eps^d$-sparse in $G$.
	\end{lemma}
	\begin{proof}
		We claim that $d\ge40$ given by \cref{lem:house4} satisfies the lemma.
		Let $x:=\eps^{5d}\in(0,2^{-d})$; and we may assume that $\abs G\ge \eps^{-10d^2}=x^{-2d}$.
		For every induced subgraph $F$ of $G$ with $\abs F\ge \eps^d\abs G$, we have
		$\abs F\ge \eps^dx^{-2d}\ge x^{-d}$;
		and so by the choice of $d$, there exists $k\in[2,1/x]$ such that there is a pure or $x$-sparse $(k,\abs F/k^d)$-blockade in $F$.
		\cref{thm:key} now gives an $(\eps^{-1},{x^{2d}\abs G})$-blockade $(B_1,\ldots,B_{\ell})$ in $G$, such that for all distinct $i,j\in[\ell]$, $(B_i,B_j)$ is either complete or weakly $\eps^d$-sparse in $G$.
		Since $x^{2d}=\eps^{10d^2}$, this proves \cref{lem:epsone}.
	\end{proof}
	This completes the first half of the proof of \cref{thm:main}.
	\section{Deducing \cref{thm:main}}
	\label{sec:2ndround}
	In this section we complete the proof of \cref{thm:main}. 
	Let us make one point which might clarify why we need two rounds of iterative sparsification. \cref{lem:epsone} gives us blockades
	with the property that every pair of blocks is complete or weakly sparse: let us call them ``semisparse'' for this discussion.
	\cref{lem:house2} tells us essentially that:
	\begin{itemize}
		\item If $G$ is $\overline{P_5}$-free and $O(y^3)$-sparse, then either we can sparsify further or there is
		a semisparse blockade of length at least $(1/y)^{1/4}$ and at most $1/x$.
	\end{itemize}
	That result passed through the machinery of iterative sparsification, and was converted to \cref{lem:epsone}.
	As explained in \cref{sec:sketch}, semisparse blockades are insufficient for us to deduce the \erh{} property of $\overline{P_5}$-free graphs immediately.
	Nevertheless, since every $\overline{P_5}$-free graph contains such a blockade (with no sparsity
	condition), we have the freedom to specify the length of the blockade, by choosing $1/\eps$ appropriately.
	In particular, we can apply \cref{lem:epsone} in a $y$-sparse graph, choosing $\eps$ to be some huge power of $y$; and we deduce that:
	\begin{itemize}
		\item 
		If $G$ is $\overline{P_5}$-free and $y$-sparse, then either we can sparsify further or there is
		a semisparse blockade of length a huge power of $1/y$.
	\end{itemize}
	This is a much more powerful version of \cref{lem:house2}, since the length of the blockade is now ``fixed'' in terms of the density parameter $y$.
	It gives rise to a new way to sparsify, that is the second round of sparcification and the key to the remainder of the proof of \cref{thm:main}.
	Our plan is to say that in such a semisparse blockade, either there is a block containing a complete blockade with appropriate length and width, or there is one containing a decent portion that is anticomplete to nearly all of the rest of $G$.
	This is done via the following lemma.
	\begin{lemma}
		\label{lem:house6}
		There exists $d\ge40$ such that the following holds.
		Let $y\in(0,\frac12)$, and let $G$ be a $y$-sparse $\overline{P_5}$-free graph.
		Then either:
		\begin{itemize}
			\item there exists $S\subset V(G)$ with $\abs S\ge y^{30d^3}\abs G$ such that $G[S]$ is $y^{2d}$-sparse;
			
			\item there is a complete $(y^{-1},y^{33d^3}\abs G)$-blockade in $G$; or
			
			\item there are disjoint $X,Y\subset V(G)$ such that $\abs X\ge y^{33d^3}\abs G$, $\abs Y\ge(1-3y)\abs G$, and $Y$ is anticomplete to $X$ in $G$.
		\end{itemize}
	\end{lemma}
	
	Let us give a sketch of the proof of this lemma, which uses 
	the semisparse blockades given by \cref{lem:robust0} (that is, \cref{lem:epsone}).
	We are given a small positive variable $y$ and a $y$-sparse $\overline{P_5}$-free graph $G$. As discussed above, we try to do sparsification;
	if we can find a slightly smaller value $y'$ such that there is a $y'$-sparse induced subgraph of size $\poly(y'/y)|G|$,
	we will take that as an outcome. 
	We apply \cref{lem:epsone} with $\eps=y^d$ to get a $(y^{-d},\floor{y^{10d^3}\abs G})$-blockade $\mac B=(B_1,\ldots,B_{\ell})$ in $G$ (where $\ell=\ceil{y^{-d}}$) such that every pair $(B_i,B_j)$ is either complete or weakly $y^{d^2}$-sparse.
	Here, unless the second outcome of \cref{lem:house6} occurs, \cref{lem:cores} and a probabilistic argument allow us to assume that each $B_i$ is anticonnected in $G$ and of size about $y^{10d^3}\abs G$ (up to minor changes in their sizes and the density between them).
	How does the rest of $G$ attach to $\mac B$? Let $v$ be some vertex not in any of the blocks of $\mac B$. Then $v$
	is anticomplete to some of the blocks, complete to others, and mixed on the remainder.
	If there is some $v$ outside of $\mac B$ that is mixed on at least $y\ell$ blocks,
	then no two of these blocks are complete to each other;
	for otherwise there would be a copy of $\overline{P_5}$; this is where the complete property is crucial 
	(see \cref{fig:conductor}).

	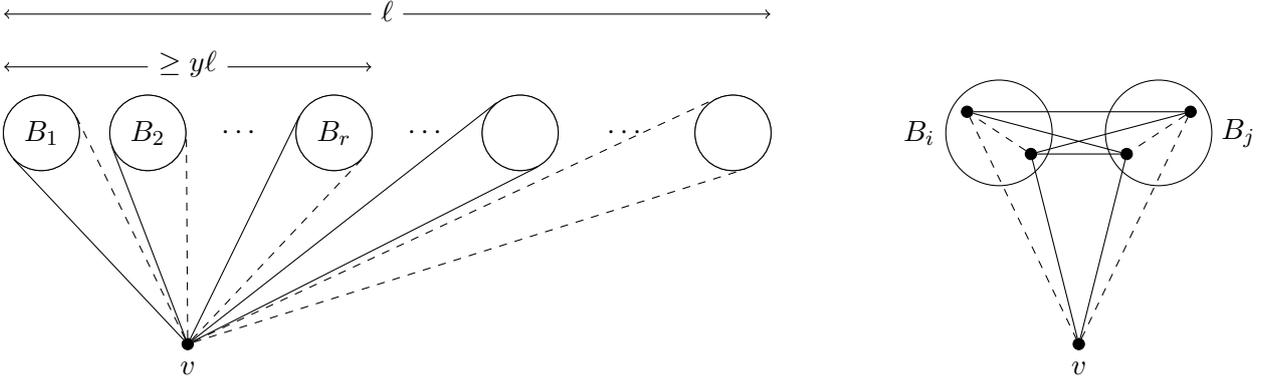
\begin{figure}[ht]
		\centering
		
		\begin{tikzpicture}[scale=0.65,auto=left]
			
			\node[inner sep=1.5pt, fill=black,circle,draw] (v) at ({-1.25},{-4}) {};
			\node[below=0.1cm] at (v) {$v$};

			\node[] (b1) at (-4,0) {};
			\node [circle,draw,name path=circle,fill=white] (c1) at (b1) [minimum size=1cm] {};
			\draw[dashed] (v)  
			-- (tangent cs:node=c1,point={(v)},solution=1) 
			coordinate[] ();
			\draw[-] (v) -- (tangent cs:node=c1,point={(v)},solution=2)
			coordinate[] ();
			\node [circle,draw,name path=circle,fill=white] (c1) at (b1) [minimum size=1cm] {};
			\node[] at (b1) {$B_1$};
			
			\node[] (b2) at (-2,0) {};
			\node [circle,draw,name path=circle,fill=white] (c2) at (b2) [minimum size=1cm] {};
			\draw[dashed] (v)  
			-- (tangent cs:node=c2,point={(v)},solution=1) 
			coordinate[] ();
			\draw[-] (v) -- (tangent cs:node=c2,point={(v)},solution=2)
			coordinate[] ();
			\node [circle,draw,name path=circle,fill=white] (c2) at (b2) [minimum size=1cm] {};
			\node[] at (b2) {$B_2$};
			
			\node[] (br) at (1.5,0) {};
			\node [circle,draw,name path=circle,fill=white] (cr) at (br) [minimum size=1cm] {};
			\draw[dashed] (v)  
			-- (tangent cs:node=cr,point={(v)},solution=1) 
			coordinate[] ();
			\draw[-] (v) -- (tangent cs:node=cr,point={(v)},solution=2)
			coordinate[] ();
			\node [circle,draw,name path=circle,fill=white] (cr) at (br) [minimum size=1cm] {};
			\node[] at (br) {$B_r$};
			
			\node[] (bi) at (5,0) {};
			\node [circle,draw,name path=circle,fill=white] (ci) at (bi) [minimum size=1cm] {};
			\draw[-] (v)  
			-- (tangent cs:node=ci,point={(v)},solution=1) 
			coordinate[] ();
			\draw[-] (v) -- (tangent cs:node=ci,point={(v)},solution=2)
			coordinate[] ();
			\node [circle,draw,name path=circle,fill=white] (ci) at (bi) [minimum size=1cm] {};
			
			\node[] (bj) at (9,0) {};
			\node [circle,draw,name path=circle,fill=white] (cj) at (bj) [minimum size=1cm] {};
			\draw[dashed] (v)  
			-- (tangent cs:node=cj,point={(v)},solution=1) 
			coordinate[] ();
			\draw[dashed] (v) -- (tangent cs:node=cj,point={(v)},solution=2)
			coordinate[] ();
			\node [circle,draw,name path=circle,fill=white] (cj) at (bj) [minimum size=1cm] {};
			
			\node[] at (-1.25,1.3) {$\ge y\ell$};
			\draw[->] (-2,1.25) -- (-4.7, 1.25) {};
			\draw[->] (-0.5,1.25) -- (2.2, 1.25) {};
			\node[] at (2.5,2.3) {$\ell$};
			\draw[->] (2.2,2.25) -- (-4.7, 2.25) {};
			\draw[->] (2.8,2.25) -- (9.7, 2.25) {};
			
			\node[] at (-0.25,0) {$\cdots$};
			\node[] at (3.25,0) {$\cdots$};
			\node[] at (7,0) {$\cdots$};
			
			\draw[] (14,0) circle (1cm);
			\draw[] (17,0) circle (1cm);
			\node[] at (12.5,-0) {$B_i$};
			\node[] at (18.5,0) {$B_j$};
			
			\node[inner sep=1.5pt, fill=black,circle,draw] (u) at ({15.5},{-4}) {};
			\node[below=0.1cm] at (u) {$v$};
			
			\node[inner sep=1.5pt, fill=black,circle,draw] (v1) at ({13.4},{0.4}) {};
			\node[inner sep=1.5pt, fill=black,circle,draw] (v2) at ({14.6},{-0.4}) {};
			\node[inner sep=1.5pt, fill=black,circle,draw] (v3) at ({16.4},{-0.4}) {};
			\node[inner sep=1.5pt, fill=black,circle,draw] (v4) at ({17.6},{0.4}) {};
			\draw[dashed] (v2) -- (v1) -- (u) -- (v4) -- (v3);
			\draw[-] (v2) -- (u) -- (v3) -- (v2) -- (v4) -- (v1) -- (v3);
		\end{tikzpicture}
		
		\caption{Using a really long semisparse blockade.}
		\label{fig:conductor}
	\end{figure}
	
	Hence, these $y\ell$ blocks are pairwise weakly $y^{d^2}$-sparse; and so their union has edge density about $O((y\ell)^{-1})=O(y^{d-1})$ and size at least $y^{10d^3}\abs G$, which is a desirable sparsification outcome.
	So we assume that there is no such $v$. It follows that 
	there is some $B_i$ with at most $O(y)\abs G$ vertices of $G$ mixed on it. But only a few vertices are complete to $B_i$
	since $G$ is $y$-sparse; so almost all are anticomplete to $B_i$. More exactly, 
	$B_i$ is anticomplete to a vertex subset of size $(1-O(y))\abs G$, 
	which satisfies the third outcome of \cref{lem:house6} since $\abs{B_i}$ is about $y^{10d^3}\abs G$.
	(This type of argument also appears in~\cite{density6} where we show that graphs of bounded VC-dimension have polynomial-sized cliques or stable sets.)

	We now provide a rigorous proof of \cref{lem:house6}, as follows.
	
	\begin{proof}
		[Proof of \cref{lem:house6}]
		We claim that $d\ge40$ given by \cref{lem:epsone} satisfies the lemma.
		To show this, let $y,G$ be as in the lemma statement;
		and assume that the first two outcomes do not hold.
		In particular $\abs G\ge y^{-30d^3}$ since the first outcome does not hold.
		Let $\eps:=y^{3d}\in (0,2^{-3d})$;
		then $\abs G\ge y^{-30d^3}= \eps^{-10d^2}$.
		Let $\ell:=\ceil{\eps^{-1}}$ and $m:=\ceil{\eps^{10d^2}\abs G}\le \eps\abs G$.
		\begin{claim}
			\label{claim:clean}
			There is a blockade $(B_1,\ldots,B_{\ell})$ in $G$ such that:
			\begin{itemize}
				\item for all $i\in[\ell]$, $B_i$ is anticonnected in $G$ and $\abs{B_i}=\ceil{\eps^2m}$; and
				
				\item for all distinct $i,j\in[\ell]$, $(B_i,B_j)$ is either complete or $\eps^{d-8}$-sparse to each other in $G$.
			\end{itemize}
		\end{claim}
		\begin{subproof}
			By \cref{lem:epsone}, there is an $(\eps^{-1},{\eps^{10d^2}\abs G})$-blockade $(A_1,\ldots,A_{\ell})$ in $G$, where $\ell=\ceil{\eps^{-1}}\le2\eps^{-1}$, such that for all distinct $i,j\in[\ell]$, $(A_i,A_j)$ is complete or weakly $\eps^d$-sparse in $G$.
			Let $J$ be the~graph with vertex set $[\ell]$ where distinct $i,j\in V(J)$ are adjacent in $J$ if and only if $A_i$ is complete to $A_j$ in $G$.
			
			For each $i\in[\ell]$, let $X_i$ be a uniformly random subset of $A_i$ of size $m=\ceil{\eps^{10d^2}\abs G}$.
			For all distinct $i,j\in[\ell]$ with $ij\nin E(J)$, the expected number of edges between $X_i,X_j$ in $G$ is at most $\eps^d\abs{X_i}\abs{X_j}$;
			and so, since $\frac12\ell^2=\frac12\ceil{\eps^{-1}}^2\le\eps^{-2}$, with positive probability $(X_i,X_j)$ is weakly $\eps^{d-2}$-sparse for all distinct $i,j\in[\ell]$ with $ij\nin E(J)$.
			
			For $i=1,2,\ldots,\ell$ in turn, define a subset $B_i$ of $X_i$ as follows.
			Assume that $B_1,\ldots,B_{i-1}$ have been defined, such that $\abs{B_p}=\ceil{\eps^2m}$ for all $1\le p<q\le\ell$ with $pq\nin E(J)$ and $p<i$,
			\begin{itemize}
				\item $B_p$ is $\eps^{d-6}$-sparse to $B_q$ and $B_q$ is $\eps^{d-8}$-sparse to $B_p$ if $q<i$; and
				
				\item 
				$B_p$ is $\eps^{d-4}$-sparse to $X_q$ if $q\ge i$.
			\end{itemize}
			For each $p\in[\ell]\setminus\{i\}$ with $pi\nin E(J)$, let $C_p$ be the set of vertices in $X_i$ with at least $\eps^{d-8}\abs{B_p}$ neighbours in $B_p$ if $p<i$,
			and let $C_p$ be the set of vertices in $X_i$ with at least $\eps^{d-4}\abs{X_p}$ neighbours in $X_p$ if $p>i$;
			then $\abs{C_p}\le \eps^2\abs{X_i}$ for all $p\in[\ell]\setminus\{i\}$.
			Let $D_i:=X_i\setminus(\bigcup_{p\in[\ell]\setminus\{i\},pi\nin E(J)}C_p)$;
			then $\abs{D_i}\ge (1-\eps^2\ell)\abs{X_i}\ge (1-2\eps)\abs{X_i}\ge\frac12m$.
			If $G[D_i]$ has no anticonnected component of size at least $\abs{D_i}/\ell$,
			then \cref{lem:cores} (with $k=\ell$) would give a complete $(\ell,\abs{D_i}/\ell^2)$-blockade in $G[D_i]$;
			but this satisfies the second outcome of the lemma since $\abs{D_i}/\ell^2\ge \frac18\eps^2m\ge \frac1{8}\eps^{2+10d^2}\abs G\ge \eps^{11d^2}\abs G=y^{33d^3}\abs G$ and $\ell\ge\eps^{-1}\ge y^{-1}$, a contradiction.
			Thus, $G[D_i]$ has an anticonnected component $B_i$ with $\abs{B_i}\ge\abs{D_i}/\ell\ge \frac14\eps m\ge\eps^2m$.
			By removing vertices from $B_i$ if necessary, we may assume that $\abs{B_i}=\ceil{\eps^2m}$.
			For every $1\le p<i$ with $pi\nin E(J)$, since $B_p$ is $\eps^{d-4}$-sparse to $X_i$, it follows that $B_p$ is $\eps^{d-6}$-sparse to $B_i$; and $B_i$ is $\eps^{d-8}$-sparse to $B_p$ by definition.
			
			This completes the inductive definition of $B_1,\ldots,B_{\ell}$; and it is not hard to check that $(B_1,\ldots,B_{\ell})$ is a blockade of $G$ satisfying the claim.
			This proves \cref{claim:clean}.
		\end{subproof}
		Let $B:=V(G)\setminus(B_1\cup\cdots\cup B_{\ell})$; then since $\eps\le y^2$, we have
		\[\abs B\ge \abs G-\ell\ceil{\eps^2m}
		\ge \abs G-2\ell\eps^2m
		\ge \abs G-4\eps m
		\ge \abs G-m\ge (1-\eps)\abs G\ge(1-y^2)\abs G.\]
		
		\begin{claim}
			\label{claim:mixed}
			No vertex in $B$ is mixed on at least $y\ell$ blocks among $(B_1,\ldots,B_{\ell})$.
		\end{claim}
		\begin{subproof}
			Suppose there is such a vertex $v\in B$;
			and assume that it is mixed on $B_1,\ldots,B_r$, where $r\ge y\ell\ge y^{2d+1}$.
			If there are distinct $i,j\in[r]$ such that $B_i$ is complete to $B_j$ in $G$,
			then since $B_i,B_j$ are anticonnected in $G$,
			there would be $u_i,w_i\in B_i$ and $u_j,w_j\in B_j$ such that $u_iv,u_jv\in E(G)$ and $w_iv,w_jv\nin E(G)$;
			but then $\{v,u_i,u_j,v_i,v_j\}$ would form a copy of $\overline{P_5}$ in $G$ (see \cref{fig:conductor}), a contradiction.
			Thus, $B_i$ is $\eps^{d-8}$-sparse to $B_j$ for all distinct $i,j\in[r]$.
			Let $S:=\bigcup_{i\in[r]}B_i$; then $\abs S=rm$ and $G[S]$ has maximum degree at most
			\[m+r\eps^{d-8}m\le (y^{2d+1}+\eps^{d-8})rm
			\le 2y^{2d+1}rm\le y^{2d}rm=y^{2d}\abs S\]
			where the penultimate inequality holds since $\eps^{d-8}=y^{3d(d-8)}\le y^{3d}\le y^{2d+1}$ (note that $d\ge40$).
			Thus $G[S]$ is $y^{2d}$-sparse; but then $S$ satisfies the first outcome of the lemma since $\abs S=rm\ge \eps^{10d^2}\abs G=y^{30d^2}\abs G$, a contradiction.
			This proves \cref{claim:mixed}.
		\end{subproof}
		\cref{claim:mixed} says that every vertex in $B$ is mixed on fewer than $y\ell$ blocks among $(B_1,\ldots,B_{\ell})$;
		and so there exists $i\in[\ell]$
		such that there are fewer than $y\abs B$ vertices in $B$ mixed on $B_i$.
		Thus, since $G$ is $y$-sparse, there are at most $y\abs G+y\abs B$ vertices in $B$ with a neighbour in $B_i$.
		Let $Y$ be the set of vertices in $B$ with no neighbour in $B_i$; then, because $\abs B\ge (1-y^2)\abs G$, we have
		\[\abs Y\ge (1-y)\abs B-y\abs G\ge (1-y)(1-y^2)\abs G-y\abs G
		\ge (1-3y)\abs G\]
		and the third outcome of the lemma holds since $\abs{B_i}\ge \eps^2m\ge \eps^{2+10d^2}\abs G\ge \eps^{11d^2}\abs G=y^{33d^2}\abs G$.
		This proves \cref{lem:house6}.
	\end{proof}
	Let us now turn the third outcome of \cref{lem:house6} into an anticomplete blockade outcome.
	\begin{lemma}
		\label{lem:house7}
		There exists $d\ge40$ such that the following holds.
		Let $y\in(0,4^{-6}]$, and let $G$ be a $y$-sparse $\overline{P_5}$-free graph. Then either:
		\begin{itemize}
			\item there exists $S\subset V(G)$ with $\abs S\ge y^{16d^3}\abs G$ such that $G[S]$ is $y^d$-sparse; or
			
			\item there is a complete or anticomplete $(y^{-1/2},y^{18d^3}\abs G)$-blockade in $G$.
		\end{itemize}
	\end{lemma}
	\begin{proof}
		We claim that $d\ge40$ given by \cref{lem:house6} satisfies the lemma.
		We may assume $\abs G\ge y^{-16d^3}$, for otherwise the first outcome trivially holds.
		Let $n\ge0$ be maximal such that there is an anticomplete blockade $(B_0,B_1,\ldots,B_n)$ of $G$ with $\abs{B_n}\ge(1-2y^{1/2})^n\abs G$ and $\abs{B_{i-1}}\ge y^{18d^3}\abs G$ for all $i\in[n]$.
		If $n\ge y^{-1/2}$ then the second outcome of the lemma holds; and so we may assume $n<y^{-1/2}$. Then since $y\le 4^{-6}$,
		\[\abs{B_n}\ge(1-3y^{1/2})^n\abs G\ge 4^{-3y^{1/2}n}\abs G\ge 4^{-3}\abs G\ge y\abs G\ge y^{-15d^3}=(y^{-1/2})^{30d^3}\]
		and so $G[B_n]$ has maximum degree at most $y\abs G\le 4^{3}y\abs {B_n}\le y^{1/2}\abs {B_n}$ since $y\le 4^{-6}$.
		Thus, by \cref{lem:house6} (with $y^{1/2}$ in place of $y$), either:
		\begin{itemize}
			\item there exists $S\subset B_n$ with $\abs S\ge y^{15d^3}\abs {B_n}$ such that $G[S]$ is $y^d$-sparse;
			
			\item there is a complete $(y^{-1/2},y^{17d^3}\abs{B_n})$-blockade in $G[B_n]$; or
			
			\item there are disjoint $X,Y\subset B_n$ such that $\abs X\ge y^{17d^3}\abs{B_n}$, $\abs Y\ge(1-2y^{1/2})\abs{B_n}$, and $Y$ is anticomplete to $X$ in $G$.
		\end{itemize}
		
		If the first bullet holds, then $\abs S\ge y^{15d^3}\abs{B_n}\ge y^{16d^3}\abs G$ and the first outcome of the lemma holds.
		If the second bullet holds, then since $y^{17d^3}\abs{B_n}\ge y^{18d^3}\abs G$, the second outcome of the lemma holds.
		If the third bullet holds, then since $\abs X\ge y^{17d^3}\abs{B_n}\ge y^{18d^3}\abs G$ and $\abs Y\ge(1-2y^{1/2})\abs{B_n}\ge(1-2y^{1/2})^{n+1}\abs G$,
		$(B_0,B_1,\ldots,B_{n-1},X,Y)$ would contradict the maximality of $n$.
		This proves \cref{lem:house7}.
	\end{proof}
	Next we eliminate the sparsity hypothesis of \cref{lem:house7}, by means of R\"odl's \cref{thm:rodl} and iterative sparsification. We deduce \cref{lem:house0}, which we restate:
	\begin{lemma}
		\label{lem:house}
		There exists $a\ge1$ such that the following holds.
		For every $x\in(0,\frac12)$ and every $\overline{P_5}$-free graph $G$, either:
		\begin{itemize}
			\item $G$ has an $x$-restricted induced subgraph with at least $x^a\abs G$ vertices; or
			
			\item there is a complete or anticomplete $(k,\abs G/k^a)$-blockade in $G$, for some $k\in[2,1/x]$.
		\end{itemize}
	\end{lemma}
	\begin{proof}
		Let $c:=4^{-6}$ and $\eta=2^{-5}$.
		Let $d\ge40$ be given by \cref{lem:house7}.
		By \cref{thm:rodl}, there exists $t\ge 36d^2$ such that for every $\overline{P_5}$-free graph $G$, there exists $S\subset V(G)$ with $\abs S\ge c^t\abs G$ such that $G[S]$ is $c$-restricted.
		We shall prove that every $a\ge2dt$ with $2^{a}\ge(\eta c^t)^{-1}$ satisfies the lemma.
		To show this, let $x\in(0,c)$, and let $G$ be $\overline{P_5}$-free.
		If $\abs G<x^{-a}$ then the first outcome of the lemma holds and we are done;
		and so we may assume $\abs G\ge x^{-a}\ge\eta^{-1}$.
		Assume that the second outcome of the lemma does not hold; that is, there is no $k\in[2,1/x]$ such that there is a complete or anticomplete $(k,\abs G/k^a)$-blockade in $G$.
		By the choice of $\theta$, there is a $c$-restricted $S\subset V(G)$ with $\abs S\ge\theta\abs G$.
		If $\overline G[S]$ is $c$-sparse, then since $\overline G[S]$ is $P_5$-free, \cref{lem:sparse} gives an anticomplete $(2,{\eta\abs S})$-blockade in $\overline G[S]$,
		a contradiction since ${\eta\abs S}\ge {\eta c^t\abs G}\ge \abs G/2^a$ by the choice of $a$.
		Hence, $G[S]$ is $c$-sparse.
		Thus, there exists $y\in[x^d,c]$ minimal (note that $x^d< 2^{-d}<2^{-12}=c$) such that $G$ has a $y$-sparse induced subgraph $F$ with $\abs F\ge y^{t}\abs G$.
		\begin{claim}
			\label{claim:decre}
			$y<x$.
		\end{claim}
		\begin{subproof}
			Suppose not.
			By \cref{lem:house7}, either:
			\begin{itemize}
				\item $F$ has a $y^d$-sparse induced subgraph with at least $y^{16d^3}\abs F$ vertices; or
				
				\item there is a complete or anticomplete $(y^{-1/2},y^{18d^3}\abs F)$-blockade in $F$.
			\end{itemize}
			
			Note that $18d^3+t\le dt\le \frac12a$ since $d\ge2$, $t\ge 36d^2\ge \frac{18d^3}{d-1}$, and $a\ge2dt$.
			Thus, if the first bullet holds, then $G$ would have a $y^d$-sparse induced subgraph with at least $y^{16d^3}\abs F\ge y^{16d^3+t}\ge y^{dt}\abs G$ vertices,
			which contradicts the minimality of $y$ since $y^d\ge x^d$.
			If the second bullet holds, then since $y^{18d^3}\abs F\ge y^{18d^3+t}\abs G\ge y^{dt}\abs G\ge y^{a/2}\abs G$,
			there would be a complete or anticomplete $(y^{-1/2},y^{a/2}\abs G)$-blockade in $G$,
			which satisfies the second outcome of the lemma (with $k=y^{-1/2}$) because $x\le y^{1/2}\le c^{1/2}\le\frac12$, a contradiction.
			This proves \cref{claim:decre}.
		\end{subproof}
		
		Since $x^d\le y<x$, we have that $F$ is $x$-sparse and $\abs F\ge y^t\abs G\ge x^{dt}\abs G\ge x^a\abs G$.
		Thus the first outcome of the lemma holds, proving \cref{lem:house}.
	\end{proof}
	Let us now prove the polynomial R\"odl property of $P_5$ (\cref{thm:main}).
	The proof method holds under a more general setting, and is similar to and simpler in part than that of \cref{thm:key}.
	\begin{theorem}
		\label{thm:blocks}
		Let $\eps\in(0,\frac12)$ and $a\ge1$, and let $G$ be a graph.
		Assume that for every induced subgraph $F$ of $G$ with $\abs F\ge\eps^{2a}\abs G$, there exists $k\in[2,1/\eps]$ such that there is a complete or anticomplete $(k,\abs F/k^a)$-blockade in $F$.
		Then $G$ has an $\eps$-restricted induced subgraph with at least $\eps^{3a}\abs G$ vertices.
	\end{theorem}
	\begin{proof}
		A {\em cograph} is a graph with no induced four-vertex path;
		and it is well known that every $n$-vertex cograph has a clique or stable set of size at least $\sqrt n$.
		Let $q\ge1$ be a maximal integer such that there exist a cograph $J$ with vertex set $[q]$ and a pure $(q,\eps^{3a}\abs G)$-blockade $(A_1,\ldots,A_q)$ in $G$ satisfying:
		\begin{itemize}
			\item for all distinct $i,j\in [q]$, $(A_i,A_j)$ is complete in $G$ if and only if $ij\in E(J)$; and
			
			\item $\sum_{j\in [q]}\abs{A_j}^{1/a}\ge\abs G^{1/a}$.
		\end{itemize}
		\begin{claim}
			\label{claim:length}
			$q\ge\eps^{-2}$.
		\end{claim}
		\begin{subproof}
			Suppose not. We may assume $\abs{A_1}=\max_{j\in[q]}\abs{A_j}$;
			then $q \abs{A_1}^{1/a}\ge\abs G^{1/a}$
			which yields $\abs{A_1}\ge \abs G/q^a\ge \eps^{2a}\abs G$.
			Thus, the hypothesis gives $k\in[2,1/\eps]$ and a complete or anticomplete $(k,\abs {A_1}/k^a)$-blockade $\mac (B_1,\ldots,B_{\ell})$ in $G[A_1]$.
			Let $J'$ be the graph obtained from $J$ by substituting a complete or edgeless graph $K$ for vertex $1$ in $J$,
			such that $\abs K=\ell$ and $K$ is complete if and only if $(B_1,B_2)$ is complete in $G[A_1]$.
			Then $J'$ is a cograph with $\abs{J'}>q$.
			Now $\abs{B_i}\ge\abs{A_1}/k^a\ge \eps^{a}\abs{A_1}\ge\eps^{3a}\abs G$ for all $i\in V(K)$; and
			$\sum_{i\in V(K)}\abs{B_i}^{1/a}\ge k(\abs{A_1}/k^a)^{1/a}=\abs{A_1}^{1/a}$ which implies
			\[\sum_{j\in [q]\setminus\{1\}}\abs{A_j}^{1/a}
			+\sum_{i\in V(K)}\abs{B_i}^{1/a}
			\ge \sum_{j\in [q]}\abs{A_j}^{1/a}
			\ge \abs G^{1/a}.\]
			Consequently $J'$ violates the maximality of $q$, a contradiction.
			This proves \cref{claim:length}.
		\end{subproof}
		
		Since $J$ is a cograph, it has a clique or stable set $I$ with $\abs I\ge\sqrt q\ge1/\eps$.
		For every $j\in I$, let $S_j\subset A_j$ with $\abs{S_j}=\ceil{\eps^{3a}\abs G}$;
		and let $S:=\bigcup_{j\in I}S_j$.
		Then $\abs S=\abs I\cdot\abs{S_j}\ge \eps^{3a}\abs G$ for all $j\in I$.
		If $I$ is a clique in $J$, then $\overline G[S]$ has maximum degree at most $\abs S/\abs I\le\eps\abs S$;
		and if $I$ is a stable set in $J$, then $G[S]$ has maximum degree at most $\abs S/\abs I\le\eps\abs S$.
		Thus $G[S]$ is an $\eps$-restricted induced subgraph of $G$ with at least $\eps^{3a}\abs G$ vertices.
		This proves \cref{thm:blocks}.
	\end{proof}
	
	The proof of \cref{thm:main} now follows shortly.
	\begin{proof}
		[Proof of \cref{thm:main}]
		Let $a\ge1$ be given by \cref{lem:house7}.
		It suffices to show that for every $\eps\in(0,\frac12)$, every $\overline{P_5}$-free graph $G$ has an $\eps$-restricted induced subgraph with at least $\eps^{3a}\abs G$ vertices.
		Suppose not.
		By \cref{lem:house7} with $x=\eps$,
		for every induced subgraph $F$ of $G$ with $\abs F\ge \eps^{2a}\abs G$, either:
		\begin{itemize}
			\item $F$ has an $\eps$-restricted induced subgraph with at least $\eps^{a}\abs F\ge\eps^{3a}\abs G$ vertices; or
			
			\item there is a complete or anticomplete $(k,\abs F/k^a)$-blockade in $F$ for some $k\in[2,1/x]$.
		\end{itemize}
		
		Since the first bullet cannot hold by our supposition, the second bullet holds for every such induced subgraph $F$.
		Then \cref{thm:blocks} implies that $G$ has an $\eps$-restricted induced subgraph with at least $\eps^{3a}\abs G$ vertices, contrary to the supposition.
		This proves \cref{thm:main}.
	\end{proof}
	\section*{Acknowledgements}
	We are grateful to the anonymous referees for helpful comments and suggestions.
	
	\bibliographystyle{abbrv}

\end{document}